\documentclass{birkjour}

\usepackage[utf8]{inputenc}
\usepackage[T1]{fontenc}
\usepackage[english]{babel}
\usepackage{mathtools,amssymb,amsthm}
\mathtoolsset{showonlyrefs}

\usepackage{geometry}
\usepackage{color}
\usepackage[pdftex]{hyperref}

\newtheorem{theorem}{Theorem}[section]
\newtheorem{lemma}[theorem]{Lemma}
\newtheorem{proposition}[theorem]{Proposition}
\newtheorem{corollary}[theorem]{Corollary}
\theoremstyle{definition}
\newtheorem{remark}[theorem]{Remark}
\newtheorem{definition}[theorem]{Definition}

\numberwithin{equation}{section}

\begin{document}

\title[Nearly-circular periodic solutions of perturbed relativistic Kepler problems]{Nearly-circular periodic solutions \\of perturbed relativistic Kepler problems:\\the fixed-period and the 
fixed-energy problems}

\author[A.~Boscaggin]{Alberto Boscaggin}

\address{
Dipartimento di Matematica ``Giuseppe Peano''\\ 
Universit\`{a} degli Studi di Torino\\
Via Carlo Alberto 10, 10123 Torino, Italy}

\email{alberto.boscaggin@unito.it}

\author[G.~Feltrin]{Guglielmo Feltrin}

\address{
Dipartimento di Scienze Matematiche, Informatiche e Fisiche\\ 
Universit\`{a} degli Studi di Udine\\
Via delle Scienze 206, 33100 Udine, Italy}

\email{guglielmo.feltrin@uniud.it}

\author[D.~Papini]{Duccio Papini}

\address{
Dipartimento di Scienze e Metodi dell'Ingegneria\\
Universit\`{a} degli Studi di Modena e Reggio Emilia\\
Via Giovanni Amendola 2, 42122 Reggio Emilia, Italy}

\email{duccio.papini@unimore.it}

\thanks{Work written under the auspices of the Grup\-po Na\-zio\-na\-le per l'Anali\-si Ma\-te\-ma\-ti\-ca, la Pro\-ba\-bi\-li\-t\`{a} e le lo\-ro Appli\-ca\-zio\-ni (GNAMPA) of the Isti\-tu\-to Na\-zio\-na\-le di Al\-ta Ma\-te\-ma\-ti\-ca (INdAM). 
The first and second authors are supported by the INdAM-GNAMPA Project 2023 ``Analisi qualitativa di problemi differenziali non lineari'' and PRIN 2022 ``Pattern formation in nonlinear phenomena''.
\\
\textbf{Preprint -- May 2024}} 



\date{}

\dedicatory{}

\begin{abstract}
The paper studies the existence of periodic solutions of a perturbed relativistic Kepler problem of the type
\begin{equation*}
\dfrac{\mathrm{d}}{\mathrm{d}t}\left(\frac{m\dot{x}}{\sqrt{1-|\dot{x}|^{2}/c^{2}}}\right) =
-\alpha\frac{x}{|x|^{3}} + \varepsilon \, \nabla_{x} U(t,x), 
\qquad x \in \mathbb{R}^d\setminus\{0\},
\end{equation*}
with $d=2$ or $d=3$, bifurcating, for $\varepsilon$ small enough, from the set of circular solutions of the unperturbed system. Both the case of the fixed-period problem (assuming that $U$ is $T$-periodic in time) and the case of the fixed-energy problem (assuming that $U$ is independent of time) are considered.
\end{abstract}

\maketitle

\section{Introduction and statement of the main results}\label{section-1}

In this paper, we investigate the existence of periodic solutions for the nonlinear Lagrangian system
\begin{equation}\label{eq:Tfissato}
\dfrac{\mathrm{d}}{\mathrm{d}t}\left(\frac{m\dot{x}}{\sqrt{1-|\dot{x}|^{2}/c^{2}}}\right) =
-\alpha\frac{x}{|x|^{3}} + \varepsilon \, \nabla_{x} U(t,x),
\qquad x \in \mathbb{R}^d\setminus\{0\},
\end{equation}
where $d=2$ or $d=3$, $m,c,\alpha$ are positive constants,
$U \colon \mathbb{R} \times (\mathbb{R}^d\setminus\{0\}) \to \mathbb{R}$ is an external potential which is periodic in the time variable,
and $\varepsilon$ is a real parameter.
More precisely, we are interested in the existence of periodic solutions to \eqref{eq:Tfissato} which stay close, for $\varepsilon$ small enough, to circular solutions of the system with $\varepsilon = 0$,
both in the case of the fixed-period problem (that is,
$U$ is $T$-periodic in time, and we look for $T$-periodic solutions) and in the case of the fixed-energy problem (that is, $U$ is independent of time, and we look for solutions having a prescribed value of energy, independent of $\varepsilon$).

The interest for system \eqref{eq:Tfissato} lies in the fact that it can be regarded as a perturbation of the problem
\begin{equation}\label{eq:nonperturbato}
\dfrac{\mathrm{d}}{\mathrm{d}t}\left(\frac{m\dot{x}}{\sqrt{1-|\dot{x}|^{2}/c^{2}}}\right) =
-\alpha\frac{x}{|x|^{3}}, \qquad x \in \mathbb{R}^d\setminus\{0\},
\end{equation}
which in turn can be interpreted as a Kepler problem with relativistic correction. Indeed, equation \eqref{eq:nonperturbato}
appears when replacing the usual Keplerian Lagrangian of classical mechanics  
$\mathcal{L}_0^{\textrm{cl}}(x,\dot x)=\frac{1}{2}m\vert \dot x \vert^2+\alpha/\vert x \vert$ with the relativistic Lagrangian
\begin{equation}\label{defL}
\mathcal{L}_0(x,\dot x) =  mc^2 \left( 1 - \sqrt{1- \dfrac{\vert \dot x \vert^2}{c^2}} \right) + \frac{\alpha}{\vert x \vert}, 
\end{equation}
where $c$ has the role of speed of light (incidentally, notice that $\mathcal{L}_0 \to \mathcal{L}_0^{\textrm{cl}}$ in the non-relativistic limit
$c \to +\infty$), see \cite{BoDaFe-2021,BDF4} for an extensive bibliography on this subject. 
Even more naturally, problem \eqref{eq:nonperturbato} can be meant in the context of electrodynamics: indeed, it appears as a particular case of the Lorentz force equation, describing the motion of a slowly accelerated charged particle under the influence of an attractive Coulomb field, see \cite{ArBeTo-19,BDP2} and the references therein for more details and bibliography.

In \cite{BoDaFe-2021}, the perturbed problem \eqref{eq:Tfissato} was considered in the planar case $d=2$. More precisely, it was shown therein that it is possible to detect $T$-periodic solutions which stay close to non-circular (``rosetta''-shaped) $T$-periodic solutions of the unperturbed problem. Notice that, whenever one of such non-circular solutions, say $x^*$, exists, by the autonomous nature and rotational invariance of the unperturbed problem, actually a whole manifold of solutions can be found: precisely, the two-dimensional torus made up of all time-translation and (planar) rotations of $x^*$. 
With this in mind, the result in \cite{BoDaFe-2021} can be interpreted in a bifurcation spirit: that is, in spite of the fact that the unperturbed periodic torus is generically destroyed by the perturbation, some periodic solutions of the unperturbed problem still ``survive'' for $\varepsilon \neq 0$ and small enough. 
Later, in \cite{BDF4}, an analogous theorem was established in the context of the fixed-energy problem. 

The above mentioned results were obtained by using action-angle coordinates, allowing to interpret problem \eqref{eq:Tfissato} as a nearly integrable Hamiltonian system, and the manifolds of solutions as unperturbed invariant periodic tori, with a fixed value of energy and angular momentum. Within this approach, bifurcation was obtained as a consequence of the fact that the unperturbed Hamiltonian satisfies some kind of KAM-type non-degeneracy condition. It is worth mentioning that the possibility of extending this technique to the spatial case $d=3$ appears rather unclear, since the required non-degeneracy condition fails in three dimensions. We also mention \cite{BDP2,BDP1} for non-perturbative results, obtained via variational methods, for $d=3$ and $d=2$, respectively. 

Motivated by the above achievements, in this paper we deal again with the perturbed problem \eqref{eq:Tfissato}, focusing on the complementary case of solutions bifurcating from circular solutions of the unperturbed problem. To the best of our knowledge, such an issue is basically unexplored in the literature. More precisely, the only result we are aware of is \cite[Corollary 9]{Ga-19}, dealing with the fixed-period problem for \eqref{eq:Tfissato} when $d=2$. However, such a result was designed for a larger class of systems and yields highly non-optimal conclusions when applied to the specific case of \eqref{eq:Tfissato} (in particular, a symmetry condition on the external potential $U$ was assumed in \cite{Ga-19}). Here, instead, we provide a systematic investigation of the problem, both in the case of the fixed-period problem and in the case of the fixed-energy problem, both when $d=2$ and (most notably) when $d=3$.

To illustrate our results, from now on we denote, for every $\tau > 0$, by $\mathcal{M}_\tau^d$ the set of circular solutions of \eqref{eq:nonperturbato} with minimal period $\tau$. It can be proved (see Section~\ref{section-2.2}) that $\mathcal{M}_\tau^d$ is actually a compact manifold. Precisely, if $x_{\tau}$ stands for a fixed circular solution of \eqref{eq:nonperturbato} with minimal period $\tau$, for $d=2$ the set $\mathcal{M}_\tau^d$ is a manifold diffeomorphic to the orthogonal group $O(2)$ and, in particular, has two connected components: the one given by the solutions of positive angular momentum (having the form $x_\tau(t-\theta)$, provided $x_\tau$ is chosen with positive angular momentum) and the one given by the solutions of negative angular momentum (having the form $x_\tau(-t-\theta)$). Instead, for $d=3$, the set $\mathcal{M}_\tau^d$ is a manifold diffeomorphic to the special orthogonal group $SO(3)$, being given by the circular solutions $M x_\tau$ on varying of $M \in SO(3)$. With this in mind, we now describe our results.

\subsection{The fixed-period problem}\label{section-1.1}

Let us first focus on the fixed-period problem. Thus, we assume that the potential $U$ is $T$-periodic in the time variable. Having fixed a circular solution $x_{\tau}$ of minimal period $\tau$, we can find $T$-periodic solutions of \eqref{eq:Tfissato} bifurcating from the manifold of circular solutions with minimal period $\tau = T/k$, where $k \geq 1$ is any integer. Precisely, the following result holds true. 

\begin{theorem}\label{teo:Tfissato}
Assume that $U \colon \mathbb{R}\times (\mathbb{R}^{d}\setminus\{0\}) \to \mathbb{R}$ is a
continuous function, which is $T$-periodic with respect to the first variable and two times continuously differentiable with respect to the second variable. Fix an integer $k \geq 1$ and, when $k\ge 2$, assume also that
\begin{equation}\label{hp-Tkj}
T\neq T^{*}_{k,j}:=\frac{2\pi\alpha jk^{3}}{mc^{3}(k^{2}-j^{2})^{\frac{3}{2}}}, \quad \text{for every $j=1,\ldots,k-1$.}
\end{equation}
Then, for each $\eta >0$ there exists $\varepsilon^{*}>0$ such that for every $\varepsilon\in (-\varepsilon^{*},\varepsilon^{*})$ the following holds.
\begin{itemize}
\item If $d=2$, there exist four $T$-periodic solutions $x_{i}$ of \eqref{eq:Tfissato} and four times $\theta_{i}\in\left[0,T/k\right)$, for $i=1,2,3,4$, such that
\begin{equation*}
|x_{i}(t)-x_{T/k}(t-\theta_{i})|<\eta,
\quad
\text{for every $t\in\mathbb{R}$ and $i=1,2$,}
\end{equation*}
and
\begin{equation*}
|x_{i}(t)-x_{T/k}(-t-\theta_{i})|<\eta,
\quad
\text{for every $t\in\mathbb{R}$ and $i=3,4$.}
\end{equation*}

\item If $d=3$, there exist four $T$-periodic solutions $x_{i}$ of \eqref{eq:Tfissato} and four matrices $M_{i}\in SO(3)$, for $i=1,2,3,4$, such that
\begin{equation*}
|x_{i}(t)-M_{i}x_{T/k}(t)|<\eta,
\quad
\text{for every $t\in\mathbb{R}$.}
\end{equation*}
\end{itemize}
\end{theorem}

Let us notice that in the planar case bifurcation is obtained from both the connected components of the manifold $\mathcal{M}_{T/k}^2$, while the statement in the spatial case reflects the fact that the manifold $\mathcal{M}_{T/k}^3$ is, instead, connected. The number of bifurcating solutions is in both cases related to a topological invariant associated with the manifold, see Remark~\ref{rem-cat} for more details.

Let us also highlight the crucial difference between the case $k=1$ (in which no restrictions on $U$ and $T$ are needed) and the case $k \geq 2$, which requires the further assumption \eqref{hp-Tkj}: this (rather unexpected) condition can be explained in terms of bifurcation of non-circular solutions for the unperturbed problem, see Remark~\ref{bif-rosette}.

We finally point out that, as a standard corollary of Theorem~\ref{teo:Tfissato}, it is possible to obtain an arbitrarily large number of $T$-periodic solutions: indeed, one can apply Theorem~\ref{teo:Tfissato} for an arbitrarily large number of different choices of the integer $k$, and the fact that the corresponding solutions stay close to different unperturbed circular solutions ensures that, for $\eta$ small enough, they are all distinct.
However, since the presence of condition \eqref{hp-Tkj} makes this argument slightly more delicate, we present it in Section~\ref{section-3.4}.

\subsection{The fixed-energy problem}\label{section-1.2}

We now assume that the potential $U$ is independent on time and we focus on the fixed-energy problem. Here, and henceforth, by energy we mean the quantity
\begin{equation}\label{eq:energia}
\mathcal{E}_\varepsilon(x):= \frac{mc^{2}}{\sqrt{1-|\dot{x}^{2}|/c^{2}}} -\frac{\alpha}{|x|} - \varepsilon \, U(x),
\end{equation}
which, as it can be easily verified, is constant along solutions of \eqref{eq:Tfissato} whenever $U$ is time-independent (see also Section~\ref{sec21} for a more formal derivation of the energy via Hamiltonian formulation).
Incidentally, let us notice that the above definition is the same as the one used in \cite{BoDaFe-2021}, while it differs by the additive constant $mc^2$ from the definition in \cite{BDF4}. 

Summing up, we thus deal with the problem
\begin{equation}\label{eq:Efissata}
\begin{cases}
\, \dfrac{\mathrm{d}}{\mathrm{d}t}\left(\dfrac{m\dot{x}}{\sqrt{1-|\dot{x}|^{2}/c^{2}}}\right) =
-\alpha\dfrac{x}{|x|^{3}} +\varepsilon \, \nabla U(x) , \qquad x \in \mathbb{R}^d\setminus\{0\},
\vspace{3pt} \\
\, \mathcal{E}_\varepsilon(x)=h,
\end{cases}
\end{equation}
where $h$ is a real constant. Actually, we will take $h \in (0,mc^{2})$, since, by Lemma~\ref{lem:circolare},
a periodic solution of the unperturbed problem \eqref{eq:nonperturbato} with energy $h$ exists if and only if $h$ belongs to that range of values, the correspondence between energy and minimal period being, moreover, bijective. With this in mind, the following result holds true. 

\begin{theorem}\label{teo:Efissata}
Assume that $U\in \mathcal{C}^{\infty}(\mathbb{R}^{d}\setminus\{0\})$ and fix $h\in(0,mc^{2})$.
Let $\tau$ be the minimal period of a circular solution $x_{\tau}$ of \eqref{eq:nonperturbato} with energy $\mathcal{E}_0(x_{\tau})=h$.
Then, for each $\eta >0$ there exists $\varepsilon^{*}>0$ such that for every $\varepsilon\in (-\varepsilon^{*},\varepsilon^{*})$ the following holds.
\begin{itemize}
\item If $d=2$, there exists a periodic solution $x$ of \eqref{eq:Efissata} with period $\sigma$ such that $|\sigma-\tau|<\eta$ and
\begin{equation*}
|x(t)-x_{\tau}(t)|<\eta, \quad \text{for every $t\in[0,\tau+\eta]$.}
\end{equation*}

\item If $d=3$, there exists a periodic solutions $x$ of \eqref{eq:Efissata} with period $\sigma$ and a matrix $M\in SO(3)$ such that  
$|\sigma-\tau|<\eta$ and
\begin{equation*}
 |x(t)-M x_{\tau}(t)|<\eta, \quad \text{for every $t\in[0,\tau+\eta]$.}
\end{equation*}
\end{itemize}
\end{theorem}

Comparing Theorem~\ref{teo:Efissata} with Theorem~\ref{teo:Tfissato} for the fixed-period case, some remarks are in order. 
First, we observe that here bifurcation from the sole manifold $\mathcal{M}_\tau^d$ must be considered: the fixed-energy problem is indeed of purely geometrical nature, and solutions are considered on their minimal period. Accordingly, there is no need of a condition like \eqref{hp-Tkj}. Second, we notice that only one solution is obtained, for both $d=2$ and $d=3$: this is due to the fact that, whenever $x(t)$ is a solution of \eqref{eq:Efissata}, the functions 
$x(\pm t + \theta)$ are also solutions (to be considered geometrically equivalent) and it seems impossible to provide geometrically distinct solutions within our approach. 

Finally, we point out that the result for $d=2$ appears of rather different nature with respect to the one for $d=3$ (as well as
to the results in the fixed-period problem, for both $d=2$ and $d=3$). Indeed, as already mentioned, each of the connected components of the manifold $\mathcal{M}_\tau^2$ is topologically a circle: this is the ``minimal complexity'' for a set of periodic solutions of an autonomous problem, since, as just observed, time-translations of solutions are still solutions. 
 Hence, rather than a bifurcation-type result, what Theorem~\ref{teo:Efissata} gives in this case is the existence of a continuation of the unperturbed periodic orbit $\{x_\tau(t)\}_{t \in \mathbb{R}}$ into a perturbed periodic orbit $\{x(t)\}_{t \in \mathbb{R}}$. See Remark~\ref{rem-continuazione} for more details. We stress that, on the contrary, the three-dimensional case is different, since the manifold $\mathcal{M}_\tau^3$ has dimension $3$ and generically disintegrates into a finite number of periodic orbits, according to the usual bifurcation scenario. To the best of our knowledge, this is the first available result for the fixed-energy problem \eqref{eq:Efissata} in dimension $d=3$.

\subsection{Comments about the proof and plan of the paper}\label{section-1.3}

While the results in \cite{BoDaFe-2021,BDF4} about bifurcation from non-circular periodic solutions (in the planar case) were obtained by passing to action-angle coordinates and checking the KAM-nondegeneracy of the unperturbed problem, the proofs of Theorem~\ref{teo:Tfissato} and Theorem~\ref{teo:Efissata} require different strategies: indeed, usual action-angle coordinates are useless in a neighborhood of circular solutions, since, already in the planar case, energy and angular momentum do not provide independent first integrals, cf.~\cite[Chapter 3, Appendix C]{BeFa-notes}. 

More precisely, for the proof of Theorem~\ref{teo:Tfissato} about the fixed-period problem we rely on a well-established abstract bifurcation result of variational nature, providing critical points of a perturbed functional when a non-degenerate manifold of critical points of the unperturbed functional exists (see, for instance, \cite[Theorem~10.8]{MaWi-89}). Here non-degeneracy is meant in a very natural sense, and the application of the abstract theorem eventually requires to compute the dimension of the space of $T$-periodic solutions of the linearization of problem \eqref{eq:nonperturbato} around a circular solution (that is, the fixed points of the associated monodromy matrix). This strategy has been used in several other papers dealing with various models of classical mechanics, see \cite{AmCo-89,BA-2023,FoGa-18} and the references therein. We point out, however, that
the technical details of our proof are quite different with respect to the ones in the above mentioned papers. Indeed, the variational formulation  of the relativistic system \eqref{eq:Tfissato} must be obtained relying on its Hamiltonian structure, rather than on the Lagrangian one (see Remark~\ref{lag-ham} for more details); moreover, the proof of the non-degeneracy condition is obtained by performing a suitable symplectic transformation (based on polar coordinates when $d=2$ and on spherical coordinates when $d=3$) which allows to greatly simplify the computations.

As for the proof of Theorem~\ref{teo:Efissata} about the fixed-energy problem, instead, we apply an abstract bifurcation theory from periodic manifolds of Hamiltonian systems, developed by Weinstein in a series of papers \cite{We-73, We-77, We-78} (see also \cite{BDF4,BoOrZh-19} for an overview and some comments). We point out that the notion of non-degeneracy in the context of the fixed-energy problem is quite delicate and actually, even within the above mentioned theory, several non-equivalent formulations have been proposed. For our application, we rely on the weakest possible notion of non-degeneracy, considered in \cite{We-73, We-78}. This eventually requires us to consider again the 
linearization of problem \eqref{eq:nonperturbato} around a circular solution; however, differently from the fixed-period problem, 
a more delicate property of the associated monodromy matrix (rather than the dimension of the space of its fixed points) has to be
proved (see the discussion following Theorem~\ref{teo-weinstein0} for the details). As in the case of solutions with fixed period, we take advantage of the change of variables to polar/spherical coordinates to provide this proof. 

\smallskip

The plan of the paper is the following.
In Section~\ref{section-2}, we recall the Hamiltonian formulation of our problem and provide a change of variables which will simplify computations; we also describe in details the manifold of circular $\tau$-periodic solutions of the unperturbed problem.
Section~\ref{section-3} deals with the fixed-period problem or, more precisely, with the proof of Theorem~\ref{teo:Tfissato} and its consequences.
In a similar way, Section~\ref{section-4} is devoted to the proof of Theorem~\ref{teo:Efissata} about the fixed-energy problem.
For the reader's convenience, we summarize in Appendix~\ref{section-appendix} the way in which the monodromy operator is transformed after a change of variables, in order to justify the computations we made to check the non-degeneracy conditions.

\section{Preliminaries}\label{section-2}

In this section, we collect some preliminary facts and results about problem \eqref{eq:Tfissato} which will be used along the paper. More precisely, in Section~\ref{sec21} we provide the Hamiltonian formulation of \eqref{eq:Tfissato} and, based on this, a further symplectic change of variables to polar coordinates if $d=2$ and spherical coordinates if $d=3$. Then, in Section~\ref{section-2.2} we focus on the unperturbed problem \eqref{eq:nonperturbato}, providing a complete description of the set of circular solutions both for $d=2$ and for $d=3$.

\subsection{The Hamiltonian formulation and some useful changes of coordinates}\label{sec21}

As already mentioned in the Introduction, equation \eqref{eq:Tfissato}
has a Lagrangian structure; however, for our purposes, it is
essential to consider the equivalent  Hamiltonian formulation. This can be done (via Legendre transformation) defining the momentum variable
\begin{equation}\label{eq-momentop}
p=\dfrac{m\dot{x}}{\sqrt{1-\dfrac{|\dot{x}|^2}{c^2}}}
\end{equation}
and the Hamiltonian 
\begin{equation}\label{eq-hamiltoniana}
\mathcal{H}_\varepsilon(t,x,p)=mc^2\sqrt{1+\dfrac{|p|^2}{m^2c^2}}-\dfrac{\alpha}{|x|} - \varepsilon \, U(t,x).
\end{equation}
In this way,  equation \eqref{eq:Tfissato} turns out to be equivalent to the $d$ degrees of freedom Hamiltonian system
\begin{equation}\label{eq-hs00}
\dot x = \nabla_p \mathcal{H}_\varepsilon(t,x,p), \qquad
\dot p = - \nabla_x \mathcal{H}_\varepsilon(t,x,p),
\end{equation}
which explicitly reads as
\begin{equation}\label{eq-hs}
\begin{cases}
\, \dot x = \dfrac{p}{m \sqrt{1 + \dfrac{| p |^2}{m^2c^2}}}, \vspace{6pt}\\
\, \dot p = - \alpha \, \dfrac{x}{|x|^3} + \varepsilon \,\nabla_x U(t,x).
\end{cases}
\end{equation}
When the Hamiltonian $\mathcal{H}_\varepsilon$ does not depend on time, it is of course a first integral of system
\eqref{eq-hs}, and expressing $p$ in terms of $\dot x$ via \eqref{eq-momentop} provides the energy function defined as in \eqref{eq:energia}. 

From now on, we focus on the unperturbed system (i.e., $\varepsilon = 0$)
\begin{equation}\label{eq-hs0}
\begin{cases}
\, \dot x = \dfrac{p}{m \sqrt{1 + \dfrac{| p |^2}{m^2c^2}}}, \vspace{6pt}\\
\, \dot p = - \alpha \, \dfrac{x}{|x|^3},
\end{cases}
\end{equation}
and we introduce a symplectic change of variables, based on polar coordinates when $d=2$ and on spherical coordinates when $d=3$, cf.~\cite[Section~7.4 and Section~7.5]{MHO-09}, which will be very useful to simplify the computations in the next sections.

Let us first assume $d=2$. In this case, the change of variables is given by 
\begin{equation*}
\Psi \colon  (\mathbb{R}^2\setminus\{0\})\times \mathbb{R}^2 \to (0,+\infty)\times \mathbb{T}^1 \times \mathbb{R}^2, \qquad 
(x,p) \mapsto (r,\vartheta,l,\Phi), 
\end{equation*}
where $ \mathbb{T}^1 = \mathbb{R}/2\pi\mathbb{Z}$, $x = (x_1,x_2)$ is written in polar coordinates as 
\begin{equation*}
x=r (\cos\vartheta,\sin\vartheta)
\end{equation*}
and, with $p=(p_1,p_2)$,
\begin{equation}\label{def-l-Phi}
l = \frac{\langle x,p \rangle}{r}, \qquad \Phi = x_1 p_2 - x_2 p_1.
\end{equation}
The map $\Psi$ is symplectic, meaning that
\begin{equation*}
\mathrm{d}x \wedge \mathrm{d}p = \mathrm{d}r \wedge \mathrm{d}l + \mathrm{d}\vartheta \wedge \mathrm{d}\Phi. 
\end{equation*}
As a consequence, system \eqref{eq-hs0} is transformed into the Hamiltonian system
\begin{equation}\label{eq-hs2}
\begin{cases}
\, \dot r = \partial_l \mathcal{K}_0(r,\vartheta,l,\Phi) = \dfrac{l}{m}\dfrac{1}{\sqrt{1+\dfrac{l^2+\Phi^2/r^2}{m^2c^2}}}, 
\vspace{7pt}\\
\, \dot \vartheta = \partial_\Phi \mathcal{K}_0(r,\vartheta,l,\Phi) = \dfrac{\Phi}{mr^{2}}\dfrac{1}{\sqrt{1+\dfrac{l^2+\Phi^2/r^2}{m^2c^2}}}, 
\vspace{7pt}\\
\, \dot l = - \partial_r \mathcal{K}_0(r,\vartheta,l,\Phi) = \dfrac{\Phi^{2}}{m r^{3}} \dfrac{1}{\sqrt{1+\dfrac{l^2+\Phi^2/r^2}{m^2c^2}}} - \dfrac{\alpha}{r^{2}}, 
\vspace{7pt}\\
\, \dot \Phi = -\partial_\vartheta \mathcal{K}_0(r,\vartheta,l,\Phi) = 0, \vspace{2pt}
\end{cases}
\end{equation}
corresponding to the Hamiltonian
\begin{equation}\label{def-H0}
\mathcal{K}_0(r,\vartheta,l,\Phi)=mc^2\sqrt{1+\dfrac{l^2+\Phi^2/r^2}{m^2c^2}}-\dfrac{\alpha}{r}.
\end{equation}
Incidentally, notice that since $\dot\Phi = 0$, the quantity $\Phi = x_1 p_2 - x_2 p_1$ is a first integral, 
typically named the (scalar) angular momentum. 

Let us now assume $d=3$. In this case, we define
\begin{equation}\label{eq-cambiovar4}
\begin{aligned}
\Psi \colon  (\mathbb{R}^3\setminus\{x_1 = x_2 = 0\})\times \mathbb{R}^3 &\to (0,+\infty)\times \mathbb{T}^1 \times (0,\pi) \times \mathbb{R}^3, 
\\ 
(x,p) &\mapsto (r,\vartheta,\varphi,l,\Phi,f), 
\end{aligned}
\end{equation}
where $x = (x_1,x_2,x_3)$ is written in spherical coordinates as 
\begin{equation*}
x=r (\sin\varphi\cos\vartheta,\sin\varphi\sin\vartheta,\cos\varphi) 
\end{equation*}
and, with $p=(p_1,p_2,p_3)$,
\begin{equation*}
l = \frac{\langle x,p \rangle}{r}, 
\quad \Phi = x_1 p_2 - x_2 p_1, 
\quad f = r \left( p_1 \cos\varphi \cos\vartheta + p_2 \cos\varphi \sin\vartheta - p_3 \sin\varphi\right).
\end{equation*}
Similarly as before, the map $\Psi$ is symplectic in the sense that
\begin{equation*}
\mathrm{d}x \wedge \mathrm{d}p = \mathrm{d}r \wedge \mathrm{d}l + \mathrm{d}\vartheta \wedge \mathrm{d}\Phi +
\mathrm{d}\varphi \wedge \mathrm{d}f,
\end{equation*}
and, accordingly, system \eqref{eq-hs0} is transformed into the Hamiltonian system
\begin{equation}\label{eq-hs3}
\left\{\begin{array}{lclcl}
\dot r &=& \partial_l \mathcal{K}_0(r,\vartheta,l,\Phi) = \dfrac{l}{m}\dfrac{1}{\sqrt{1+\dfrac{l^2+f^2/r^2+\Phi^2/(r^2 \sin^2 \varphi)}{m^2c^2}}}, 
\vspace{7pt}\\
\dot \vartheta &=& \partial_\Phi \mathcal{K}_0(r,\vartheta,l,\Phi) = \dfrac{\Phi}{mr^{2} \sin^2\vartheta}\dfrac{1}{\sqrt{1+\dfrac{l^2+f^2/r^2+\Phi^2/(r^2 \sin^2 \varphi)}{m^2c^2}}}, \vspace{7pt}\\
\dot \varphi &=& \partial_f \mathcal{K}_0(r,\vartheta,l,\Phi) = \dfrac{f}{mr^2}\dfrac{1}{\sqrt{1+\dfrac{l^2+f^2/r^2+\Phi^2/(r^2 \sin^2 \varphi)}{m^2c^2}}}, \vspace{7pt}\\
\dot l &=& - \partial_r \mathcal{K}_0(r,\vartheta,l,\Phi) 
\vspace{7pt}\\
&=& \left( \dfrac{f^{2}}{m r^{3}} +\dfrac{\Phi^{2}}{m r^{3}\sin^2\vartheta} \right) \dfrac{1}{\sqrt{1+\dfrac{l^2+f^2/r^2+\Phi^2/(r^2 \sin^2 \varphi)}{m^2c^2}}} - \dfrac{\alpha}{r^{2}},
\vspace{7pt}\\
\dot \Phi &=& -\partial_\vartheta \mathcal{K}_0(r,\vartheta,l,\Phi) = 0, \vspace{7pt}\\
\dot f &=& - \partial_\varphi \mathcal{K}_0(r,\vartheta,l,\Phi) = \dfrac{\Phi^{2}\cos\varphi}{m r^{3} \sin^3\varphi} \dfrac{1}{\sqrt{1+\dfrac{l^2+f^2/r^2+\Phi^2/(r^2 \sin^2 \varphi)}{m^2c^2}}}, \vspace{2pt}
\end{array}
\right.
\end{equation}
corresponding to the Hamiltonian
\begin{equation}\label{defK03}
\mathcal{K}_0(r,\vartheta,\varphi,l,\Phi,f)=mc^2\sqrt{1+\dfrac{l^2+f^2/r^2+\Phi^2/(r^2 \sin^2 \varphi)}{m^2c^2}}-\dfrac{\alpha}{r}.
\end{equation}
Notice that, as in the case $d=2$, the variable $\Phi = x_1 p_2 - x_2 p_1$ is still a first integral, corresponding to the third component of the vector angular momentum $x \wedge p$ (which is itself a first integral, as it is easy to check). We also stress that, in our exposition, we have  assumed $\varphi \in (0,\pi)$, that is, the position $x \in \mathbb{R}^3$ does not belong to the $x_3$-axis, cf.~\eqref{eq-cambiovar4}. This is not a serious problem, because, by the conservation of the angular momentum $x \wedge p$, 
solutions $(x(t),p(t))$ of system \eqref{eq-hs0} are such that $x(t)$ belongs to a plane. In particular,
choosing $(\varphi,f) \equiv (\pi/2,0)$ provides the equatorial dynamics of system \eqref{eq-hs3}, which actually gives rise to system
\eqref{eq-hs2}.

\subsection{The manifold of circular solutions for the unperturbed problem}\label{section-2.2}

Let us consider the unperturbed problem
\begin{equation}\label{eq:nonperturbato2}
\dfrac{\mathrm{d}}{\mathrm{d}t}\left(\frac{m\dot{x}}{\sqrt{1-|\dot{x}|^{2}/c^{2}}}\right) =
-\alpha\frac{x}{|x|^{3}}, \qquad x \in \mathbb{R}^d\setminus\{0\},
\end{equation}
where, as usual, $d=2$ or $d=3$. 

\begin{definition}\label{defmanifold}
Given $\tau > 0$, we denote by 
$\mathcal{M}_\tau^d$ the set of circular (i.e., $\vert x(t) \vert \equiv R$ for some $R > 0$) non-constant periodic solutions to \eqref{eq:nonperturbato2} with minimal period $\tau$.
\end{definition}

Incidentally, we notice that, as it will be clear from the proof of Lemma~\ref{lem:circolare} below and the subsequent discussion, any circular solution to \eqref{eq:nonperturbato2} is actually non-constant and periodic (with constant angular velocity); therefore, from now on we will simply refer to $\mathcal{M}_\tau^d$ as the set of circular solutions to \eqref{eq:nonperturbato2} with minimal period $\tau$.

The rest of this section is devoted to clarify the structure of $\mathcal{M}_\tau^d$ (here meant as a subset of $\tau$-periodic continuous functions, with the topology of uniform convergence), which actually turns out to be a compact manifold, both for $d=2$ and $d=3$. We begin with the following basic result, dealing with the planar case.

\begin{lemma}\label{lem:circolare}
Let $d=2$. For every $\tau > 0$, there is a unique (up to planar rotations) circular solution $x_\tau$ of \eqref{eq:nonperturbato2} having minimal period $\tau$ and positive angular momentum, given by
\begin{equation}\label{xtau1}
x_\tau(t) = R \, (\cos(\omega t),\sin(\omega t)),
\end{equation}
where $\omega = 2\pi/\tau$ is the frequency and
$R >0$ is the unique solution of the equation
\begin{equation}\label{eq-lem:circolare}
m \omega^2  R^3 = \frac{\alpha}{c} \sqrt{c^2 - \omega^2 R^2}.
\end{equation}
Moreover, the correspondence $\tau \mapsto h(\tau)$ between period and energy is a strictly increasing homeomorphism from $(0,+\infty)$ onto the open interval $(0,mc^2)$.
\end{lemma}

\begin{proof}
Let us first observe that any circular solution of \eqref{eq:nonperturbato2} is non-constant and periodic, with constant angular velocity: indeed, from \eqref{eq-hs2} it is immediately seen that, if $r = \vert x \vert$ is constant, then $\dot \vartheta$ is constant, as well.
Hence, a circular solution $x_{\tau}$ of \eqref{eq:nonperturbato2} having minimal period $\tau$ and positive (scalar) angular momentum has the form \eqref{xtau1}, for some $R > 0$ and $\omega = 2\pi/\tau$. From equation \eqref{eq:nonperturbato2}, one easily obtains that $R$ is the unique solution of \eqref{eq-lem:circolare}.

In order to find the correspondence between $\tau$ and the energy $h$, we first
look for the relation between $R$ and $h$.
To do this, we notice on the one hand that, since $\ell=0$ by \eqref{def-l-Phi}, formula \eqref{def-H0} gives
\begin{equation*}
h = mc^2\sqrt{1+\dfrac{L^2}{m^2c^2R^{2}}}-\dfrac{\alpha}{R},
\end{equation*}
where $L > 0$ is the constant value of the angular momentum $\Phi$, 
and thus
\begin{equation}\label{eq_HR}
(h^{2} - m^{2} c^{4})R^{2} + 2\alpha h R + \alpha^{2} - L^2c^{2} = 0.
\end{equation}
On the other hand, from the analysis in \cite[p.~5820]{BoDaFe-2021}, we have that
\begin{equation}\label{analisiBDF}
0 < h < mc^{2}, \qquad \dfrac{\alpha^{2}}{c^{2}} < L^{2} = \dfrac{\alpha^{2}m^{2} c^{2}}{m^{2}c^{4}-h^{2}}.
\end{equation}
Therefore, simple computations show that equation \eqref{eq_HR} has the unique solution
\begin{equation*}
R = \dfrac{\alpha h}{m^{2}c^{4}-h^{2}}.
\end{equation*}
Notice that $R$ is strictly increasing as a function of $h$,
with $R\to 0$ as $h\to0$, and $R\to+\infty$ as $h\to mc^{2}$.

At this point, from \eqref{eq-lem:circolare} we find
\begin{equation*}
\omega^{2} = \dfrac{1}{2R^{6}} \left( - \biggl{(}\dfrac{\alpha R}{c}\biggr{)}^{\!2} + \sqrt{\biggl{(}\dfrac{\alpha R}{c}\biggr{)}^{\!4} + 4 \alpha^{2} R^{6}} \right)
\end{equation*}
We observe that the right-hand side of the above formula, as a function of $R$, is a strictly decreasing homeomorphism of $(0,+\infty)$. Since $\tau=2\pi/\omega$, we conclude that $(0,+\infty) \ni \tau \mapsto h(\tau) \in (0,mc^2)$ is a strictly increasing homeomorphism.
\end{proof}

\begin{remark}\label{relazioneLR}
For further convenience, we observe that, for the circular solution $x_\tau$, 
the radius $R$ can be written in terms of the angular momentum $L$ as
\begin{equation}\label{formulaLR}
R = \dfrac{L \sqrt{L^{2}c^{2}-\alpha^{2}}}{\alpha mc}.
\end{equation}
Indeed, from \eqref{eq-momentop} and \eqref{def-l-Phi} we find
\begin{equation}\label{formulaLR2}
L = \dfrac{m c \omega R^{2}}{\sqrt{c^2 - \omega^2 R^2}}
\end{equation}
and a combination of the above formula with \eqref{eq-lem:circolare} gives
\begin{equation}\label{formulaomegaR}
\omega = \dfrac{\alpha}{LR}.
\end{equation} 
Replacing into \eqref{formulaLR2} finally yields \eqref{formulaLR}.
\hfill$\lhd$
\end{remark}

From Lemma~\ref{lem:circolare}, an explicit characterization of 
$\mathcal{M}_\tau^d$ can be deduced, both in dimension $d=2$ and $d=3$ (compare to the discussion in \cite{BA-2023}, in the case of a Kepler problem with non-Newtonian potential).

Let us first deal with the case $d=2$. In this setting, on the one hand all circular solutions of \eqref{eq:nonperturbato2} with period $\tau$ and positive angular momentum are obtained as planar rotations of the solution $x_\tau(t)$ defined in \eqref{xtau1} (hence, being of the type $x_\tau(t-\theta)$ for some $\theta \in [0,\tau)$); on the other hand, circular solutions of minimal period $\tau$ and negative angular momentum can be obtained by a reflection with respect to the $x_1$-axis of the solution $x_\tau(t)$ (giving rise to the solution $x_\tau(-t)$) possibly followed by a planar rotation. 
Thus, the manifold $\mathcal{M}_\tau^2$ is homeomorphic to the orthogonal group $O(2)$ and, in particular, has two connected components: the one given by the solutions of positive angular momentum (having the form $x_\tau(t-\theta)$) and the one given by the solutions of negative angular momentum (having the form $x_\tau(-t-\theta)$).

We now focus on the more delicate case $d=3$. In this situation, we first observe that any solution of \eqref{eq:nonperturbato2} necessarily lies on a plane. Thus, we can apply Lemma~\ref{lem:circolare} to infer that the initial position $x(0)$ and the initial momentum $p(0)$ (cf.~\eqref{eq-momentop}) of any circular solution with minimal period $\tau$ satisfy
\begin{equation}\label{ci}
\vert x(0) \vert = R \quad \text{ and } \quad 
\vert p(0) \vert = \frac{mc \omega R}{\sqrt{c^2-\omega^2 R^2}}.
\end{equation}
To describe the global topology of the manifold 
$\mathcal{M}_\tau^3$, we then
observe that a vector $(x(0),p(0))$ satisfying \eqref{ci} actually gives rise to a circular solution if and only if 
\begin{equation}\label{ortogonali}
\langle x(0),p(0) \rangle = 0.
\end{equation} 
Combining \eqref{ci} and \eqref{ortogonali}, we thus infer that the manifold $\mathcal{M}_\tau^3$ (which is clearly homeomorphic to the set of initial conditions $(x(0),p(0))$) is homeomorphic to the unit tangent bundle of $\mathbb{S}^2$, which in turn (cf.~\cite[\S~III.1.4]{CuBa-15} for more details) is homeomorphic to the special orthogonal group $SO(3)$, via the map
\begin{equation*}
(x(0),p(0)) \mapsto M = \text{column} \left(\frac{x(0)}{\vert x(0) \vert},\frac{p(0)}{\vert p(0) \vert},\frac{x(0) \wedge p(0)}{\vert x(0) \wedge p(0)\vert}\right) \in SO(3).
\end{equation*}
In particular, if we fix the reference solution
\begin{equation}\label{xtau2}
x_\tau(t) = R(\cos (\omega t),\sin (\omega t),0),
\end{equation}
with $\omega$ and $R$ as in Lemma~\ref{lem:circolare}, the matrix $M$ corresponding to $x_\tau$ is simply $M=\mathrm{Id}_{\mathbb{R}^3}$.
From this, using the fact that
$M x \wedge My = M (x \wedge y)$ for every $x,y \in \mathbb{R}^3$ and $M \in SO(3)$, it is easy to see that, given $M \in SO(3)$, the matrix corresponding to the circular solution $M x_\tau(t)$ is exactly $M$. Hence, the manifold $\mathcal{M}_\tau^3$ can be characterized simply as the set of 
$\{ M x_\tau(t) \colon M \in SO(3) \}$.

Summing up, we have the following result. Note the slight abuse of notation, since $x_\tau$ now stands both for the planar circular solution given by \eqref{xtau1} and for the spatial circular solution given by \eqref{xtau2}, the correct interpretation being however clear from the context. 

\begin{proposition}\label{topologiaM}
The following holds true.
\begin{itemize}
\item The set $\mathcal{M}^2_\tau$ is a manifold homeomorphic to $O(2)$ and, precisely, 
\begin{equation*}
\mathcal{M}^2_\tau = \mathcal{M}^2_{\tau,+} \cup \mathcal{M}^2_{\tau,-},
\end{equation*}
where $\mathcal{M}^2_{\tau,\pm}$ are the connected components of $\mathcal{M}^2_\tau$ and are given by 
\begin{equation*}
\mathcal{M}^2_{\tau,+} 
= \bigl\{ M x_\tau(t) \colon M \in SO(2) \bigr\}
= \bigl\{ x_\tau (t - \theta) \colon \theta \in [0,\tau)\bigr\}
\end{equation*}
and
\begin{equation*}
\mathcal{M}^2_{\tau,-} 
= \bigl\{ M x_\tau(t) \colon M \in O(2) \setminus SO(2) \bigr\}
= \bigl\{ x_\tau (-t - \theta) \colon \theta \in [0,\tau) \bigr\}.
\end{equation*}
\item The set $\mathcal{M}^3_\tau$ is a manifold homeomorphic to $SO(3)$ and, precisely,
\begin{equation*}
\mathcal{M}^3_\tau = \bigl\{ M x_\tau(t) \colon M \in SO(3)\bigr\}.
\end{equation*}
\end{itemize}
\end{proposition}

\begin{remark}\label{rem-2.5}
Since in dimension $d=3$ problem \eqref{eq:nonperturbato2} is invariant under the full orthogonal group $O(3)$, one could wonder why, in the description of the manifold $\mathcal{M}_\tau^3$, only the special orthogonal group $SO(3)$ appears.
In fact $O(3)$ can be viewed as a trivial covering space of $\mathcal{M}_{\tau}^{3}$ by projecting each matrix $M\in O(3)$ to the unique $\tau$-periodic circular solution with
\begin{equation*}
x(0)=R c_{1} \quad\text{and}\quad p(0)=\frac{mc\omega R}{\sqrt{c^{2}-\omega^{2}R^{2}}}c_{2},
\end{equation*}
where $c_{1}$ and $c_{2}$ are the first and second columns of $M$, respectively.
Since the third column $c_{3}$ of $M\in O(3)$ can only be either $c_{1}\wedge c_{2}$ (if $M\in SO(3)$) or $-c_{1}\wedge c_{2}$ (if $M\not\in SO(3)$), we deduce that such covering has degree $2$ and is trivial since $O(3)$ is not connected.

As an example, assume we want to map the reference circular solution $x_{\tau}(t)$ in \eqref{xtau1} to $x_{\tau}(-t)$ and observe that they both lie on the plane $x_{3}=0$.
This can be done either by using a rotation $M^{+}\in SO(3)$ of $180^{\circ}$ around the $x_{1}$-axis or a reflection $M^{-}\in O(3)\setminus SO(3)$ with respect to the plane $x_{2}=0$.
On the other hand, the restrictions of $M^{+}$ and $M^{-}$ to the plane $x_{3}=0$ coincide both with a reflection in $\mathbb{R}^{2}$ with respect to the $x_{1}$-axis which lies in $O(2)\setminus SO(2)$.
\hfill$\lhd$
\end{remark}

\begin{remark}\label{rem-cat}
In the next sections, the Lusternik--Schnirelman categories of $\mathcal{M}^2_{\tau,\pm}$ and $\mathcal{M}^3_\tau$ will play a role (recall that the Lusternik--Schnirelman category of a topological space $X$ is the smallest number of closed and contractible sets needed to cover $X$). By known results, it holds that
\begin{equation*}
\mathrm{cat}\left( \mathcal{M}^2_{\tau,\pm}\right) = 2 \quad \text{ and } \quad \mathrm{cat}\left( \mathcal{M}^3_{\tau}\right) = 4,
\end{equation*}
cf.~for instance \cite[Lemma~4]{BA-2023}.
\hfill$\lhd$
\end{remark}

\section{The fixed-period problem}\label{section-3}

In this section, we provide the proof of Theorem~\ref{teo:Tfissato}, dealing with the fixed-period problem
associated with the system
\begin{equation}\label{eq:Tfissato3}
\dfrac{\mathrm{d}}{\mathrm{d}t}\left(\frac{m\dot{x}}{\sqrt{1-|\dot{x}|^{2}/c^{2}}}\right) =
-\alpha\frac{x}{|x|^{3}} +\varepsilon \, \nabla_{x} U(t,x) , \qquad x \in \mathbb{R}^d\setminus\{0\}, \\
\end{equation}
where $U \colon \mathbb{R} \times (\mathbb{R}^d \setminus \{0\}) \to \mathbb{R}$ is a
continuous function, which is $T$-periodic with respect to the first variable and two times continuously differentiable with respect to the second variable.

More precisely, in Section~\ref{section-3.1} we recall the abstract bifurcation result that we are going to use and we describe the general strategy of the proof. The verification of non-degeneracy condition, which is the key point for the application of the abstract result, is given in Section~\ref{section-3.2} and Section~\ref{section-3.3}, dealing respectively with the case $d=2$ and $d=3$. Finally, in Section~\ref{section-3.4} we present a corollary  of Theorem~\ref{teo:Tfissato}, providing abundance of $T$-periodic solutions of problem
\eqref{eq:Tfissato3}.

\subsection{The abstract bifurcation result and the strategy of the proof}\label{section-3.1}

As anticipated in the Introduction, the proof of Theorem~\ref{teo:Tfissato} will be based on a well-established abstract perturbation theorem of variational nature (see \cite[Theorem~2.1]{AmCoEk-87} and the references therein), which we recall here for the reader's convenience, in the version stated in \cite[Theorem~10.8]{MaWi-89}.

\begin{theorem}\label{thastratto1}
Let $H$ be a real Hilbert space, $\Omega \subset H$ be an open set, and $\mathcal{A}_\varepsilon \colon \Omega \to \mathbb{R}$
be a family of twice continuously differentiable functions depending smoothly on $\varepsilon$.
Moreover, let $\mathcal{M} \subset \Omega$ be a compact manifold (without boundary) such that:
\begin{itemize}
\item[$(i)$] $\mathrm{d}\mathcal{A}_0(z) = 0$, for every $z \in \mathcal{M}$;
\item[$(ii)$] $\mathrm{d}^2 \mathcal{A}_0(z)$ is a Fredholm operator of index zero, for every $z \in \mathcal{M}$;
\item[$(iii)$] $T_z \mathcal{M} = \mathrm{ker} (\mathrm{d}^2 \mathcal{A}_0(z))$, for every $z \in \mathcal{M}$.
\end{itemize}
Then, for every neighborhood $U$ of $\mathcal{M}$, 
there exists $\varepsilon^*> 0$ such that, 
if $\vert \varepsilon \vert < \varepsilon^*$, the functional
$\mathcal{A}_\varepsilon$ has at least $\mathrm{cat}(\mathcal{M})$ critical points in $U$.
\end{theorem}

Some simple remarks about this theorem are in order. In particular, we notice that the three assumptions $(i)$, $(ii)$ and $(iii)$ all deal with the unperturbed functional $\mathcal{A}_0$. More precisely, condition $(i)$ simply requires that the manifold $\mathcal{M}$ is made up of critical points of $\mathcal{A}_0$ (that is, it is a so-called critical manifold). The assumptions $(ii)$ and $(iii)$, instead, are concerned with the second differential $\mathrm{d}^2 \mathcal{A}_0$ evaluated at points $z$ of the critical manifold. In particular, $(ii)$ is a structural assumption for 
$\mathrm{d}^2 \mathcal{A}_0(z)$ (meant, according to Riesz representation theorem, as a self-adjoint linear bounded operator on $H$), while $(iii)$ is a non-degeneracy condition for the critical manifold $\mathcal{M}$. In this regard, it is worth noticing that the inclusion $T_z \mathcal{M} \subset \mathrm{ker} (\mathrm{d}^2 \mathcal{A}_0(z))$ always holds true. Therefore, condition $(iii)$ is equivalent to the fact that the dimension of $\mathrm{ker} (\mathrm{d}^2 \mathcal{A}_0(z))$ coincides with the dimension of $\mathcal{M}$ (and, thus, is the least possible).

In order to prove Theorem~\ref{teo:Tfissato}, we are going to apply Theorem~\ref{thastratto1} to the Hamiltonian action functional associated with \eqref{eq:Tfissato3} (see Remark~\ref{lag-ham} for some comments about this choice). More precisely, let us recall that equation \eqref{eq:Tfissato3} can be equivalently written in Hamiltonian form as system \eqref{eq-hs00}, where the Hamiltonian $\mathcal{H}_\varepsilon$ is given by \eqref{eq-hamiltoniana}. Formally, this system admits the variational formulation
\begin{equation*}
\mathcal{A}_\varepsilon(z) = \mathcal{Q}(z) - \int_0^T \mathcal{H}_\varepsilon(t,x(t),p(t))\,\mathrm{d}t, 
\end{equation*}
where $z = (x,p)$ and $ \mathcal{Q}(z) =\int_0^T \langle p(t), \dot x(t) \rangle \,\mathrm{d}t$. As discussed, for instance, in \cite[Chapter 3]{Ab-01} a convenient rigorous formulation of this variational principle can be obtained (under some assumptions on the Hamiltonian) by working in the fractional Sobolev space
\begin{equation*}
H:= H^{\frac{1}{2}}_T = \biggl\{ z = (x,p) \in L^2 ( \mathbb{R}/T\mathbb{Z}, \mathbb{R}^{2d}) \colon 
\sum_{k \in \mathbb{Z}} \vert k \vert \vert \hat{z}_k \vert^2 < +\infty\biggr\},
\end{equation*}
where $\hat{z}_k \in \mathbb{C}^{2d}$ are the complex Fourier coefficients of the function $z$ (in this setting, the term $\mathcal{Q}(z)$
is meant as the unique continuous quadratic form obtained by extension, via the density of smooth functions in $H$, 
of the integral $\int_0^T \langle p(t), \dot x(t) \rangle \,\mathrm{d}t$, see again \cite{Ab-01} for more details).
  
In our setting, however, the presence of the singularity ($x= 0$) poses an extra difficulty: indeed, given any $z = (x,p) \in H$ with $x(t) \neq 0$ for every $t$, the functional $\mathcal{A}_\varepsilon$ is not well-defined in any neighborhood of $z$, since functions in the Sobolev space $H$ are not, in general, $L^\infty$. To overcome this technical issue, we argue as follows (cf.~\cite{Ga-19} for a similar approach).

Let us consider, according to Definition~\ref{defmanifold} and Proposition~\ref{topologiaM}, the compact manifold $\mathcal{M}$ given by 
\begin{equation}\label{sceltaM}
\mathcal{M} =
\begin{cases}
\, \mathcal{M}^2_{T/k,+} \text{ or } \mathcal{M}^2_{T/k,-}, &\text{if $d=2$,} \\
\, \mathcal{M}^3_{T/k}, & \text{if $d=3$.}
\end{cases}
\end{equation}
With a slight abuse of notation, here and throughout this section, the above objects will be meant as set of functions $z = (x,p)$ where $x$ is a circular solution to \eqref{eq:nonperturbato2} of minimal period $T/k$ and $p$ is given by \eqref{eq-momentop}, so that $z$ is a solution of \eqref{eq-hs0}. Clearly, $\mathcal{M} \subset H$ and, with the topology induced by the Hilbert norm, the homeomorphisms stated in Proposition~\ref{topologiaM} still hold true. 

Again by Proposition~\ref{topologiaM}, for any $z = (x,p) \in \mathcal{M}$ it holds that
$\vert x(t) \vert= R$ for every $t \in [0,T]$, with $R > 0$ independent of $z$.
Accordingly, we take a smooth cut-off function $\chi \colon [0,+\infty) \to [0,+\infty)$ such that
\begin{equation}\label{cutoff}
\chi(s) = \begin{cases}
\, 1, &  \text{ if } s \in \bigl(\frac{R}{2},\frac{3R}{2}\bigr), \\
\, 0, & \text{ if } s \in \bigl[0, \frac{R}{4}\bigr) \cup (2R,+\infty),
\end{cases}
\end{equation}
and then we define a modified Hamiltonian $\widehat{\mathcal{H}}_{\varepsilon} \colon \mathbb{R} \times \mathbb{R}^{2d} \to \mathbb{R}$ as
\begin{equation*}
\widehat{\mathcal{H}}_{\varepsilon}(t,x,p) = mc^2\sqrt{1+\dfrac{|p|^2}{m^2c^2}} - \chi(\vert x \vert) \left( \dfrac{\alpha}{|x|} + \varepsilon \, U(t,x)\right).
\end{equation*}
Since the above function is two times continuously differentiable with respect to $(x,p)$, with globally bounded first and second derivatives, standard results 
(see, for instance, \cite[Proposition~2.1]{BaSz-05}) guarantee that the corresponding Hamiltonian action functional
\begin{equation}\label{azione-modificata}
\widehat{\mathcal{A}}_{\varepsilon}(z) = \mathcal{Q}(z) - \int_0^T \widehat{\mathcal{H}}_{\varepsilon}(t,x(t),p(t))\,\mathrm{d}t
\end{equation}
is well-defined and of class $\mathcal{C}^2$ on the whole Hilbert space $H$.

Let us not define the constant
\begin{equation*}
\mathfrak{c} = (1 + \mathfrak{c}_1) \mathfrak{c}_2 \mathfrak{c}_3, 
\end{equation*}
where $\mathfrak{c}_1> 0$ is a Lipschitz constant for the function
\begin{equation}\label{def-Omega}
\mathbb{R}^d \ni p \mapsto G(p) := \dfrac{p}{m \sqrt{1 + \dfrac{| p |^2}{m^2c^2}}},
\end{equation}
$\mathfrak{c}_2> 0$ is a constant such that
\begin{equation*}
\Vert u \Vert_{L^{\infty}} \leq \mathfrak{c}_2 \bigl( \Vert u \Vert_{L^2} + \Vert \dot u \Vert_{L^2} \bigr), 
\quad \text{for every $u \in \mathcal{C}^1([0,T],\mathbb{R}^d)$,}
\end{equation*}
and $\mathfrak{c}_3> 0$ is a constant for the embedding $H \subset L^2 ( \mathbb{R}/T\mathbb{Z}, \mathbb{R}^{2d})$.
Accordingly, we consider the open set
\begin{equation*}
U_\eta = \bigl\{ z \in H \colon \mathrm{dist}_{H}(z,\mathcal{M}) < \eta/\mathfrak{c} \bigr\}.
\end{equation*}
Then, the following crucial lemma holds true.

\begin{lemma}\label{lemma-approx}
For every $\eta \in (0,R/2)$ and for every $\varepsilon \in \mathbb{R}$, if $z = (x,p)$ is a critical point of 
$\widehat{\mathcal{A}}_{\varepsilon}$ belonging to $U_\eta$, then 
$x$ is a $T$-periodic solution of \eqref{eq:Tfissato3} which satisfies
$\Vert x - x^* \Vert_{L^\infty} < \eta$ for some $z^* = (x^*,p^*) \in \mathcal{M}$.
\end{lemma}

\begin{proof}
Let $z = (x,p)$ be a critical point of $\widehat{\mathcal{A}}_{\varepsilon}$ belonging to $U_\eta$. Then, $z$ is a $T$-periodic solution of the Hamiltonian system with Hamiltonian $\widehat{\mathcal{H}}_{\varepsilon}$ and, in particular, 
$\dot x = G(p)$, where $G$ is as in \eqref{def-Omega}. Moreover, by the compactness of $\mathcal{M}$ we infer the existence of
$z^* = (x^*,p^*) \in \mathcal{M}$ such that $\Vert z - z^* \Vert_{H} < \eta/\mathfrak{c}$. 
Again, $\dot x^* = G(p^*)$. 
Using the definition of $\mathfrak{c}_1, \mathfrak{c}_2$ and $\mathfrak{c}_3$, we thus find
\begin{align*}
\Vert x - x^* \Vert_{L^{\infty}} & \leq  \mathfrak{c}_2 \left( \Vert x- x^* \Vert_{L^2} + \Vert \dot x - \dot x^* \Vert_{L^2} \right) 
\leq \mathfrak{c}_2 \left( \Vert x- x^* \Vert_{L^2} + \mathfrak{c}_1 \Vert p -  p^* \Vert_{L^2} \right) \\
& \leq (1+\mathfrak{c}_1) \mathfrak{c}_2  \Vert z- z^* \Vert_{L^2}
\leq (1+\mathfrak{c}_1) \mathfrak{c}_2 \mathfrak{c}_3  \Vert z- z^* \Vert_{H} < \eta.
\end{align*}
Since $\vert x^*(t) \vert \equiv R$, one has, 
\begin{equation*}
\vert \vert x(t) \vert - R \vert \leq \Vert x - x^* \Vert_{L^{\infty}} < \eta < \frac{R}{2}, \quad \text{for every $t \in [0,T]$.}
\end{equation*}
Then, $R/2 < \vert x(t) \vert < 3R/2$ for every $t$, and thus, by \eqref{cutoff}, $x$ is a $T$-periodic solution of the original equation  \eqref{eq:Tfissato3}.
\end{proof}

We are now in a position to summarize. Given $T > 0$ and $k \geq 1$, we consider the manifold $\mathcal{M}$ as in \eqref{sceltaM} and the associated modified action functional \eqref{azione-modificata}. Fixed an arbitrary $\eta > 0$ (which, without loss of generality, we can take smaller than $R/2$, so that Lemma~\ref{lemma-approx} applies), we are going to prove that the assumptions of Theorem~\ref{thastratto1} are satisfied. If this is the case, we obtain that there is $\varepsilon^* = \varepsilon^*(\eta) > 0$ such that, for $\varepsilon \in (-\varepsilon^*,\varepsilon^*)$, the functional $\widehat{\mathcal{A}}_{\varepsilon}$ possesses at least $\mathrm{cat}(\mathcal{M})$ critical points in the open set $U_\eta$ (hence, two critical points in the $2d$-case and four critical points in the $3d$-case, cf.~Remark~\ref{rem-cat}). By Lemma~\ref{lemma-approx}, these critical points are actually $T$-periodic solutions of 
\eqref{eq:Tfissato3} and satisfy $\Vert x - x^* \Vert_{L^{\infty}} < \eta$ for some $z^* = (x^*,p^*) \in \mathcal{M}$. Together with the characterization of the manifold $\mathcal{M}$ given in Proposition~\ref{topologiaM}, this finally proves Theorem~\ref{teo:Tfissato} (notice that in the $2d$-case the above argument has to be used both with $\mathcal{M}^2_{T/k,+}$ and $\mathcal{M}^2_{T/k,-}$ in order to get four solutions).

As a consequence, to conclude the proof one needs to check that the assumptions $(i)$, $(ii)$ and $(iii)$ of Theorem~\ref{thastratto1} hold true for the unperturbed (modified) action functional $\widehat{\mathcal{A}}_{0}$. Condition $(i)$, namely the fact that $\mathcal{M}$ is a critical manifold, is obvious, because every function $z \in \mathcal{M}$ has range in a region where the modified Hamiltonian $\widehat{\mathcal{H}}_{0}$ coincides with the original one, thus being a critical point of both $\mathcal{A}_0$ and $\widehat{\mathcal{A}}_{0}$. Condition $(ii)$, namely that 
$\mathrm{d}^{2} \widehat{\mathcal{A}}_{0}(z)$ is a Fredholm operator for every $z \in \mathcal{M}$, is a well-known fact, see \cite[p.~78]{Ab-01}.

It thus remains to check the non-degeneracy condition $(iii)$, that is, that the dimension of $\mathcal{M}$ coincides with the dimension of $\mathrm{ker} (\mathrm{d}^{2} \widehat{\mathcal{A}}_{0}(z))$, which as well known is nothing but the space of $T$-periodic solutions of the system obtained by linearizing the Hamiltonian system with Hamiltonian $\widehat{\mathcal{H}}_{0}$ around the $T$-periodic solution $z \in \mathcal{M}$. Since the range of $z$ is contained in a region where this Hamiltonian coincides with the original one, this system is just the one obtained by linearizing the unperturbed problem \eqref{eq-hs0} around $z$. Summing up, we thus have to prove that the dimension of the space of $T$-periodic solutions of the system obtained by linearizing \eqref{eq-hs0} around $z$ equals the dimension of $\mathcal{M}$. We are going to provide this proof in the next two sections, dealing respectively with the case $d=2$ and $d=3$.
 
\begin{remark}\label{lag-ham}
One could wonder why, for the variational formulation of equation \eqref{eq:Tfissato3}, we have passed to the equivalent Hamiltonian system, instead of relying (maybe more naturally) on the original Lagrangian structure. The reason is that the Lagrangian action functional
\begin{equation*}
\mathcal{I}_\varepsilon(x) = \int_0^T \mathcal{L}_0(x(t),\dot x(t)) \,\mathrm{d}t + \varepsilon \int_0^T U(t,x(t))\,\mathrm{d}t,
\end{equation*}
where $\mathcal{L}_0$ as in \eqref{defL}, is not of class $\mathcal{C}^2$ with respect to the $H^1$-topology. This is due not only to the fact that the relativistic kinetic energy is smoothly defined only when $\vert \dot x \vert < c$ (this difficulty could be overcome by cut-off arguments as the ones in this section) but, more seriously, to the fact that, as pointed out in \cite{AbSc-01}, the Lagrangian action functional is of class $\mathcal{C}^2$ if and only if $\mathcal{L}_0$ is exactly quadratic with respect to $\dot x$, which is not the case in our relativistic setting. Working in a stronger topology (like $H^2$, for instance), on the contrary, would make the functional $\mathcal{I}_\varepsilon$ smooth, but in this case the second differential would not be a Fredholm map. 
\hfill$\lhd$
\end{remark}

\subsection{Non-degeneracy in the $2d$-case}\label{section-3.2}

We consider the case $\mathcal{M} = \mathcal{M}^2_{T/k,+}$, the other being analogous.
Moreover, without loss of generality, we assume that $z^* = (x^*,p^*) \in \mathcal{M}^2_{T/k,+}$ is the solution of 
\eqref{eq-hs0} with
\begin{equation*}
x^*(t) = R (\cos(\omega t),\sin (\omega t)),
\end{equation*}
where 
\begin{equation}\label{eq-omega-T}
\omega = \frac{2\pi k}{T}
\end{equation}
and, according to Lemma~\ref{lem:circolare}, $R$ satisfies \eqref{eq-lem:circolare}.
Since the dimension of $\mathcal{M}^2_{T/k,+}$ is equal to one, we need to check that the space of $T$-periodic solutions
of the system obtained by linearizing \eqref{eq-hs0} around $z^*$ is one-dimensional, as well.

To verify this fact, we are going to make use of the change of variables to polar coordinates described in Section~\ref{sec21}. 
Precisely, instead of considering \eqref{eq-hs0} we deal with \eqref{eq-hs2}. In these new coordinates, the solution $z^*$ reads as
\begin{equation*}
\zeta^*(t) = (R,\omega t,0,L),
\end{equation*}
where, according to Remark~\ref{relazioneLR}, formula \eqref{formulaLR} relating $L$ and $R$ holds true.
As explained in Appendix~\ref{section-appendix}, the space of $T$-periodic solutions of the system obtained by linearizing \eqref{eq-hs0} around $z^*$ and the space of $T$-periodic solutions of the system obtained by linearizing \eqref{eq-hs2} at $\zeta^{*}$ are isomorphic as linear spaces, so that we can prove the validity of the non-degeneracy condition in the new coordinates.

At this point, with a long but elementary computation we find that the system obtained by linearizing \eqref{eq-hs2} at $\zeta^{*}$ is the autonomous system 
\begin{equation}\label{sistema-lineare1}
\dot w=\mathbb{A}w, \quad w \in \mathbb{R}^4,
\end{equation}
where $\mathbb{A}$ is the $4 \times 4$ real matrix given by
\begin{equation*}
\mathbb{A} =
\begin{pmatrix}
0 & 0 & A & 0 \vspace{1pt} \\
B & 0 & 0 & C \vspace{1pt} \\
D & 0 & 0 & -B \vspace{1pt} \\
0 & 0 & 0 & 0 
\end{pmatrix},
\end{equation*}
with coefficients
\begin{align*}
&A=\dfrac{cR}{\sqrt{m^{2}c^{2}R^{2}+L^{2}}},&
&B=\dfrac{L c(L^{2}-2(m^{2}c^{2}R^{2}+L^{2}))}{R^{2} (m^{2}c^{2}R^{2}+L^{2})^{\frac{3}{2}}},
\\
&C=\dfrac{m^{2}c^{3}R}{(m^{2}c^{2}R^{2}+L^{2})^{\frac{3}{2}}},&
&D=\dfrac{L^{4}c-3L^{2}c(m^{2}c^{2}R^{2}+L^{2}) + 2 \alpha (m^{2}c^{2}R^{2}+L^{2})^{\frac{3}{2}}}{R^{3} (m^{2}c^{2}R^{2}+L^{2})^{\frac{3}{2}}}.
\end{align*}
Using \eqref{formulaLR}, the above quantities can be written in terms of the angular momentum $L$ as
\begin{equation}\label{def-ABCD}
\begin{aligned}
&A=\dfrac{\sqrt{L^2c^2-\alpha^2}}{mLc},&
&B=\dfrac{\alpha^3m^2(\alpha^2-2L^2c^2)}{L^5(L^2c^2-\alpha^2)},
\\
&C=\dfrac{\alpha^2 m\sqrt{L^2c^2-\alpha^2}}{L^5c},&
&D=-\dfrac{\alpha^4m^3c}{L^5\sqrt{L^2c^2-\alpha^2}}.
\end{aligned}
\end{equation}

As well known, the space of (initial conditions of) $T$-periodic solutions of system \eqref{sistema-lineare1} is nothing but the eigenspace of the monodromy matrix $\mathbb{P} = e^{T\mathbb{A}}$ associated with the eigenvalue $1$. Thus, to conclude the proof we need to show that $1$ is an eigenvalue of $\mathbb{P}$ with geometric multiplicity equal to one.

In order to do this, we first compute the eigenvalues of $\mathbb{A}$.
A simple computation shows that the characteristic polynomial of the matrix $\mathbb{A}$ is given by
\begin{equation*}
\lambda^2 (\lambda^2 - AD)
\end{equation*}
and hence $\mathbb{A}$ has the complex eigenvalues $\{0, 0, \pm i\omega'\}$, where
\begin{equation}\label{eq:autovalori}
\omega'=\sqrt{-AD} = \dfrac{\alpha^2m}{L^3}.
\end{equation}
First, we claim that the geometric multiplicity of $0$ as eigenvalue of $\mathbb{A}$ (that is, the dimension of the kernel of $\mathbb{A}$,
i.e., the nullity of $\mathbb{A}$) is equal to one. 
Indeed, by the rank-nullity theorem, we infer
\begin{equation*}
\mathrm{nullity}(\mathbb{A}) = 4 - \mathrm{rank}(\mathbb{A})
\end{equation*}
and the rank of $\mathbb{A}$ is exactly three, since 
\begin{equation*}
B^2 + CD  = \dfrac{\alpha^6 m^4 }{L^{10} (L^2c^2-\alpha^2)^2} \left( (2L^2c^2 - \alpha^2)^2 - (L^2c^2-\alpha^2)^2 \right) \neq 0,
\end{equation*}
where the last inequality follows from the fact that $L^2 c^2 - \alpha^2 > 0$, by \eqref{analisiBDF}.

Second, we claim that $\omega' \notin 2\pi \mathbb{Z}/T$. Indeed, if $\omega' = 2\pi j/T $ for some $j\in\mathbb{Z}$ (of course, $j \geq 1$ since $\omega' > 0$), then by \eqref{eq:autovalori}
\begin{equation}\label{eq-Lj}
T = \dfrac{2\pi j L^{3}}{\alpha^{2}m}.
\end{equation}
Next, combining \eqref{eq-omega-T}, \eqref{formulaomegaR} and \eqref{formulaLR}, we deduce
\begin{equation}\label{eq-T-3.2}
T = \dfrac{2\pi k L^{2} \sqrt{L^{2}c^{2} - \alpha^{2}}}{\alpha^{2}mc} 
\end{equation}
and so, using \eqref{eq-Lj},
\begin{equation*}
j = \dfrac{k \sqrt{L^{2}c^{2} - \alpha^{2}}}{Lc}.
\end{equation*}
Hence, $j < k$, which is of course impossible if $k=1$. If instead $k \geq 2$, then $1 \leq j \leq k-1$ and a simple computation gives
\begin{equation}\label{eq-Lkj}
L = \dfrac{\alpha k}{ c \sqrt{k^{2} - j^{2}}}.
\end{equation}
Using \eqref{eq-Lkj} in \eqref{eq-Lj} we finally conclude that
\begin{equation*}
T = \frac{2\pi\alpha jk^{3}}{mc^{3}(k^{2}-j^{2})^{\!\frac{3}{2}}},
\end{equation*}
a contradiction with hypothesis \eqref{hp-Tkj}.

We are now in a position to conclude. 
Indeed, by the first claim, the matrix $\mathbb{A}$ can be written in real Jordan form as
\begin{equation*}
J_{\mathbb{A}} = 
\begin{pmatrix}
0 & 1 & 0 & 0 \vspace{1pt} \\
0 & 0 & 0 & 0 \vspace{1pt} \\
0 & 0 & 0 & \omega' \vspace{1pt} \\
0 & 0 & -\omega' & 0 
\end{pmatrix}
\end{equation*}
and, thus, the monodromy matrix $\mathbb{P} = e^{T\mathbb{A}}$ can be written in real Jordan form as
\begin{equation*}
J_{\mathbb{P}} = 
\begin{pmatrix}
1 & T & 0 & 0 \vspace{1pt} \\
0 & 1 & 0 & 0 \vspace{1pt} \\
0 & 0 & \cos(\omega' T) & \sin(\omega' T) \vspace{1pt} \\
0 & 0 & -\sin(\omega' T) & \cos(\omega' T)
\end{pmatrix}.
\end{equation*}
By the second claim, the $2 \times 2$ block associated with $\omega'$ does not produce the eigenvalue $1$. On the other hand, by the structure of the first $2 \times 2$ block, we see that $1$ is an eigenvalue of $\mathbb{P}$ with geometric multiplicity equal to one, as desired.

\begin{remark}\label{bif-rosette}
The values $T^*_{k,j}$ appearing in hypothesis \eqref{hp-Tkj} have a precise interpretation in terms of the topological structure of the set of \emph{all} $T$-periodic solutions of the unperturbed problem \eqref{eq:nonperturbato2}: when $k\geq 2$ problem \eqref{eq:nonperturbato2} may have also \emph{rosetta $T$-periodic solutions $x$ of type $(j,k)$} which means that, as $t$ ranges in $[0,T]$, $x(t)$ makes $k$ turns around the origin and $|x|$ has minimal period $T/j$, with $j\in\{1,\dots,k-1\}$.
A solution of type $(j,k)$ exists if and only if $T>T^*_{k,j}$ (see \cite[Proposition~4.3]{BDP1}).
If we use the period $T$ as a bifurcation parameter, we can show that the $T$-periodic solution of type $(j,k)$ bifurcates from the circular solution as $T\to (T^*_{k,j})^{+}$.
More precisely, the modulus $|x|$ of a non-circular $T$-periodic solution $x$ of type $(j,k)$ tends to the radius of the $T^{*}_{k,j}/k$-periodic circular solution as $T\to (T^*_{k,j})^{+}$.

Indeed, according to \cite[Propositions~2.1 and~2.2 and formula (2.26)]{BoDaFe-2021} we call $h=\mathcal{E}_{0}(x)$ the energy of $x$ and we have that the period of $|x|$ is
\begin{equation*}
T_{h}=\frac{T}{j}=\frac{2\pi\alpha m^{2}c^{3}}{(m^{2}c^{4}-h^{2})^{\!\frac{3}{2}}},
\end{equation*}
while the range of $|x|$ is $[r_{m},r_{M}]$ where
\begin{equation*}
r_{m}=\frac{\alpha h-\sqrt{\Delta}}{m^{2}c^{4}-h^{2}}, \quad
r_{M}=\frac{\alpha h+\sqrt{\Delta}}{m^{2}c^{4}-h^{2}}, \quad\text{with}\quad
\Delta=\alpha^{2}m^{2}c^{4}+\frac{\alpha^{2}k^{2}(h^{2}-m^{2}c^{4})}{k^{2}-j^{2}}.
\end{equation*}
As a consequence, we have that
\begin{equation*}
m^{2}c^{4}-h^{2}\to \biggl(\frac{2\pi\alpha m^{2}c^{3}j}{T^{*}_{k,j}}\biggr)^{\!\frac{2}{3}}=
\frac{m^{2}c^{4}(k^{2}-j^{2})}{k^{2}},
\quad \text{as $T\to (T^*_{k,j})^{+}$,}
\end{equation*}
and, thus, $\Delta\to 0$ and $r_{m}, r_{M}$ tends to the radius of the circular solution that turns around the origin $k$ times in $[0,T^{*}_{k,j}]$.
\hfill$\lhd$
\end{remark}

\subsection{Non-degeneracy in the $3d$-case}\label{section-3.3}

Here $\mathcal{M} = \mathcal{M}^3_{T/k}$ and, up to a rotation, we can assume that $z^* = (x^*,p^*) \in \mathcal{M}^3_{T/k}$ is the solution of 
\eqref{eq-hs0} with
\begin{equation*}
x^*(t) = R (\cos(\omega t),\sin (\omega t),0),
\end{equation*}
where $\omega$ is as in \eqref{eq-omega-T} and $R$ satisfies \eqref{eq-lem:circolare}.

According to Proposition~\ref{topologiaM}, $\mathrm{dim}(\mathcal{M}^3_{T/k})=3$ and, consequently, we need to verify that the space of $T$-periodic solutions of the system obtained by linearizing \eqref{eq-hs0} around $z^*$ is three-dimensional, as well. 

We perform the change of variables \eqref{eq-cambiovar4} introduced in Section~\ref{sec21} and accordingly we deal with \eqref{eq-hs3} in place of \eqref{eq-hs0}. In the new coordinates, the solution $z^*$ reads as
\begin{equation*}
\zeta^*(t) = \biggl{(}R,\omega t, \dfrac{\pi}{2},0,L,0 \biggr{)},
\end{equation*}
recalling that $(\varphi,f) \equiv (\pi/2,0)$ due to the equatorial dynamics of $z^{*}$.
As in Section~\ref{section-3.2}, thanks to the discussion in Appendix~\ref{section-appendix}, we are going to verify the non-degeneracy condition in the new coordinates.

Via standard computations one can find that the system obtained by linearizing \eqref{eq-hs3} at $\zeta^{*}$ is the autonomous system 
\begin{equation}\label{sistema-lineare2}
\dot w=\mathbb{A}w, \quad w\in\mathbb{R}^{6},
\end{equation}
where $\mathbb{A}$ is the $6\times6$ real matrix given by
\begin{equation}\label{defA3}
\mathbb{A} =
\begin{pmatrix}
0 & 0 & 0 & A & 0 & 0 \vspace{1pt} \\
B & 0 & 0 & 0 & C & 0 \vspace{1pt} \\
0 & 0 & 0 & 0 & 0 & E \vspace{1pt} \\
D & 0 & 0 & 0 & -B & 0 \vspace{1pt} \\
0 & 0 & 0 & 0 & 0 & 0 \vspace{1pt} \\
0 & 0 & F & 0 & 0 & 0
\end{pmatrix},
\end{equation}
where $A$, $B$, $C$, $D$ are defined as in \eqref{def-ABCD} and
\begin{equation*}
E = \dfrac{c}{R \sqrt{m^{2}c^{2} R^{2} + L^{2}}},
\qquad
F = \dfrac{-cL^{2}}{R \sqrt{m^{2}c^{2} R^{2} + L^{2}}}.
\end{equation*}
Next, exploiting \eqref{formulaLR}, we have
\begin{equation*}
E = \dfrac{\alpha^{2} m c }{L^{3} \sqrt{L^{2}c^{2} - \alpha^{2}}},
\qquad 
F = -\dfrac{\alpha^{2} m c }{L \sqrt{L^{2}c^{2} - \alpha^{2}}}.
\end{equation*}

Proceeding as in Section~\ref{section-3.2}, to verify condition $(iii)$ of Theorem~\ref{thastratto1} we need to show that $1$ is an eigenvalue of the monodromy matrix $\mathbb{P} = e^{T\mathbb{A}}$ with geometric multiplicity equal to three.
Since the characteristic polynomial of the matrix $\mathbb{A}$ is given by
\begin{equation*}
\lambda^2 (\lambda^2 - AD) (\lambda^{2}-EF)
\end{equation*}
the eigenvalues of $\mathbb{A}$ are $\{0, 0, \pm i\omega'', \pm i \omega '\}$, where 
\begin{equation}\label{eq:autovalori2}
\omega''=\sqrt{-EF} = \dfrac{\alpha^2m c}{L^2 \sqrt{L^{2}c^{2}-\alpha^{2}}}
\end{equation}
and $\omega'$ is as in \eqref{eq:autovalori}. 

Since $B^2 + CD\neq0$, it is easy to prove that $\mathrm{rank}(\mathbb{A})=5$ and so $\mathrm{nullity}(\mathbb{A})=1$, by the rank-nullity theorem. This fact proves that the geometric multiplicity of $0$ as eigenvalue of $\mathbb{A}$ is equal to one. As a consequence, the matrix $\mathbb{A}$ can be written in real Jordan form as
\begin{equation*}
J_{\mathbb{A}} = 
\begin{pmatrix}
0 & 1 & 0 & 0 & 0 & 0 \vspace{1pt} \\
0 & 0 & 0 & 0 & 0 & 0 \vspace{1pt} \\
0 & 0 & 0 & \omega'' & 0 & 0 \vspace{1pt} \\
0 & 0 & -\omega'' & 0 & 0 & 0 \vspace{1pt} \\
0 & 0 & 0 & 0 & 0 & \omega' \vspace{1pt} \\
0 & 0 & 0 & 0 & -\omega' & 0 
\end{pmatrix}.
\end{equation*}
Next, using \eqref{eq-T-3.2} and \eqref{eq-omega-T} we find that
\begin{equation*}
\omega'' = \dfrac{2\pi k}{T} = \omega
\end{equation*}
and, thus, the monodromy matrix $\mathbb{P} = e^{T\mathbb{A}}$ can be written in real Jordan form as
\begin{equation*}
J_{\mathbb{P}} = 
\begin{pmatrix}
1 & T & 0 & 0 & 0 & 0 \vspace{1pt} \\
0 & 1 & 0 & 0 & 0 & 0 \vspace{1pt} \\
0 & 0 & 1 & 0 & 0 & 0 \vspace{1pt} \\
0 & 0 & 0 & 1 & 0 & 0 \vspace{1pt} \\
0 & 0 & 0 & 0 & \cos(\omega' T) & \sin(\omega' T) \vspace{1pt} \\
0 & 0 & 0 & 0 & -\sin(\omega' T) & \cos(\omega' T)
\end{pmatrix}.
\end{equation*}
As in Section~\ref{section-3.2}, the $2 \times 2$ block associated with $\omega'$ does not produce the eigenvalue $1$. On the other hand, by the structure of the $4 \times 4$ block, we finally infer that $1$ is an eigenvalue of $\mathbb{P}$ with geometric multiplicity equal to three. The verification of the non-degeneracy condition is thus complete.

\subsection{A corollary: abundance of $T$-periodic solutions}\label{section-3.4}

We now prove the following corollary of Theorem~\ref{teo:Tfissato}, which has been already sketched in the Introduction. 

\begin{corollary}\label{cor-abundance}
Let $U \colon \mathbb{R}\times (\mathbb{R}^{d}\setminus\{0\}) \to \mathbb{R}$ be a
continuous function ($d=2,3$), $T$-periodic with respect to the first variable and two times continuously differentiable with respect to the second variable.
Then, for every integer $N \geq 1$ there exists $\varepsilon^*_N > 0$ such that, if $\vert \varepsilon \vert < \varepsilon^*_N$, problem \eqref{eq:Tfissato} has at least $4N$ $T$-periodic solutions.
\end{corollary}

\begin{proof}
Let us claim that there exists a strictly increasing sequence of integers $\{k_l\}_{l \geq 1}$ such that $k_1 =1$ and, 
for every $l\geq 2$, 
\begin{equation*}
T \neq T^*_{k_l,j}, \quad \text{for every $j=1,\ldots,k_l-1$.}
\end{equation*}
Assuming that this is true, given any integer $N \geq 1$, Theorem~\ref{teo:Tfissato} can be applied for $k = k_l$ with $l=1,\ldots,N$ and, thus, for every $\eta > 0$ we
find $\varepsilon^*_{k_l} > 0$ such that, if 
$\vert \varepsilon \vert < \varepsilon^*_{k_l}$, problem 
\eqref{eq:Tfissato} admits four $T$-periodic solutions which are $\eta$-close to the manifold of circular solutions of minimal period $T/{k_l}$. Thus, if 
\begin{equation*}
\vert \varepsilon \vert < \varepsilon^*_N := \min_{l=1,\ldots,N} \varepsilon^*_{k_l}
\end{equation*}
problem \eqref{eq:Tfissato} has at least $4N$ $T$-periodic solutions, which are all distinct provided $\eta$ is chosen sufficiently small. 

It thus remains to prove the claim. In order to do this, let us assume by contradiction that there exists an integer $\bar{k}$ such that, for every $k > \bar{k}$, it holds that
\begin{equation*}
T = T^*_{k,j_k},
\end{equation*}
for some $j_k \in \{1,\ldots,k-1\}$, that is
\begin{equation}\label{jkappa}
j_k = \gamma T \frac{(k^{2}-j_k^{2})^{\frac{3}{2}}}{k^3},
\end{equation}
where we have set for simplicity $\gamma = mc^3/(2\pi\alpha)$.
So, we have
\begin{equation}\label{eq-ineq_jk}
j_k < \gamma T,
\end{equation}
implying that the sequence $\{j_k\}_{k > \bar{k}}$ is bounded. Hence, passing to the limit in \eqref{jkappa} yields
$j_k \to \gamma T$ for $k \to +\infty$ so that, since $j_k$ is an integer,
$j_k = \gamma T$ for every $k$ large enough. Therefore, a contradiction with \eqref{eq-ineq_jk} is obtained. The claim is proved.
\end{proof}

\section{The fixed-energy problem}\label{section-4}

In this section, we provide the proof of Theorem~\ref{teo:Efissata}, dealing with the fixed-energy problem
associated with the system
\begin{equation}\label{eq:Efissata-sec}
\dfrac{\mathrm{d}}{\mathrm{d}t}\left(\frac{m\dot{x}}{\sqrt{1-|\dot{x}|^{2}/c^{2}}}\right) =
-\alpha\frac{x}{|x|^{3}} + \varepsilon \, \nabla U(x),
\qquad x \in \mathbb{R}^d\setminus\{0\},
\end{equation}
where $U \colon \mathbb{R}^d \setminus \{0\} \to \mathbb{R}$ is a potential of class $\mathcal{C}^\infty$. Let us recall that the energy of a solution of 
\eqref{eq:Efissata-sec} is defined as
\begin{equation}\label{eq:energia2}
\mathcal{E}_\varepsilon(x):= \frac{mc^{2}}{\sqrt{1-|\dot{x}^{2}|/c^{2}}} -\frac{\alpha}{|x|} - \varepsilon \, U(x).
\end{equation}

More precisely, in Section~\ref{section-4.1} we recall the bifurcation result that we are going to use in this fixed-energy context and we describe the general strategy of the proof. As in the case of the fixed-period problem, a non-degeneracy condition is the main issue for the application of the abstract result: the proof that this condition is satisfied is given in Section~\ref{section-4.2} for $d=3$, and in Section~\ref{section-4.3} for $d=2$.

\subsection{The abstract bifurcation result and the strategy of the proof}\label{section-4.1}

The proof of Theorem~\ref{teo:Efissata} relies on an abstract bifurcation theory 
developed by Weinstein in a series of papers \cite{We-73, We-77, We-78}, dealing with fixed-energy solutions of the perturbed Hamiltonian system 
\begin{equation*}
\dot z = X_{\mathcal{H}_\varepsilon}(z)
\end{equation*}
on a (symplectic) manifold $\mathfrak{M}$ (in the above formula, according to the usual notation,  $X_{\mathcal{H}_\varepsilon}$ is the Hamiltonian vector field associated to the Hamiltonian  $\mathcal{H}_{\varepsilon}$, that is, $\omega(X_{\mathcal{H}_{\varepsilon}},Y) = \mathrm{d}\mathcal{H}_{\varepsilon}(Y)$ for every vector field $Y$ on $\mathfrak{M}$, where $\omega$ is the symplectic form on $\mathfrak{M}$). More precisely, we are going to make use of the following result, which corresponds to \cite[Theorem~1.4]{We-73} in the case of the action of the group $\mathbb{Z}_{n}$ for $n=1$.

\begin{theorem}\label{teo-weinstein0}
Let $(\mathfrak{M},\omega)$ be a symplectic manifold such that the form $\omega$ is exact. Let $\mathcal{H}_\varepsilon$ be a family of $\mathcal{C}^\infty$-functions on $\mathfrak{M}$ depending smoothly on $\varepsilon$ and let $h \in \mathbb{R}$.
Moreover, let $\Sigma$ be a closed submanifold of $\mathfrak{M} \times (0,+\infty)$ such that: 
\begin{itemize}
\item[$(i)$] for every $(z,\tau) \in \Sigma$ it holds that
\begin{equation*}
\varphi^{\tau}(z) = z, \qquad \mathcal{H}_{0}(z) = h, \qquad X_{\mathcal{H}_{0}}(z) \neq 0,
\end{equation*}
where $t \mapsto\varphi^t(\xi)$ denotes the solution of the Hamiltonian system $\dot z =  X_{\mathcal{H}_0}(z)$ with $\varphi^0(\xi) = \xi$;

\item[$(ii)$] the restriction to $\Sigma$ of the projection $\pi \colon \mathfrak{M} \times \mathbb{R} \to \mathfrak{M}$ is an embedding;

\item[$(iii)$] for every $(z,\tau) \in \Sigma$ and for every $v \in T_z (\mathcal{H}_{0}^{-1}(h))$ it holds that
\begin{equation*}
v \in T_z (\pi(\Sigma)) \quad \Longleftrightarrow \quad v = P v + \lambda X_{\mathcal{H}_{0}}(z), \quad \text{ for some } \lambda \in \mathbb{R},
\end{equation*}
where $P \colon T_z \mathfrak{M} \to T_z \mathfrak{M}$ is the monodromy operator at $z$ (i.e., $P = \mathrm{d}_\xi \varphi^\tau(\xi)|_{ \xi = z}$).
\end{itemize}
Then, denoting by $m$ the least integer greater than or equal to $\mathrm{cat}(\Sigma)/2$, for every neighborhood $\mathcal{U} \subset \mathfrak{M} \times (0,+\infty)$ of $\Sigma$, there exists $\varepsilon^{*} > 0$ such that,
if $\vert \varepsilon \vert < \varepsilon^{*}$, there exist $\sigma_{1},\ldots,\sigma_m > 0$ and there exist 
$z_{1},\ldots,z_m$ geometrically distinct periodic solutions of the fixed-energy problem
\begin{equation*}
\begin{cases}
\, \dot z = X_{\mathcal{H}_\varepsilon}(z), \\
\, \mathcal{H}_\varepsilon(z) = h,
\end{cases}
\end{equation*}
such that, for every $i=1,\ldots,m$, $\sigma_i$ is a period for $z_i$ and 
\begin{equation}\label{loc-theorem}
\bigl\{(z_i(t),\sigma_i) \colon t \in \mathbb{R} \bigr\} \subset \mathcal{U}.
\end{equation}
\end{theorem}

Some remarks about assumptions and conclusions of Theorem~\ref{teo-weinstein0} are in order. 
As for the assumptions, we first notice that condition $(i)$ requires that, for every $(z,\tau) \in \Sigma$, the path 
$t \mapsto \varphi^t(z)$ is a non-constant $\tau$-periodic solution of the unperturbed Hamiltonian system $\dot z =  X_{\mathcal{H}_0}(z)$ with energy equal to $h$ (that is, $\Sigma$ is a so-called periodic manifold). Incidentally, notice that $\tau$ does not need to be the minimal period of this solution; however, condition $(ii)$ ensures that the choice of the period $\tau$ is unique and smooth with respect to the initial condition $z$. Finally, condition $(iii)$ is a non-degeneracy condition for the periodic manifold $\Sigma$. 
It is worth noticing that this condition is different from the one considered in the fixed-period case, which was concerned with the dimension of the space of fixed points of the monodromy operator $P$ (compare with the proofs in Section~\ref{section-3.2} and Section~\ref{section-3.3}). 
Actually, the notion of non-degeneracy for the fixed-energy problem is a quite delicate issue, and several non-equivalent definitions have been considered in the literature (see \cite[Remark~2.1]{BDF4} for more comments).
With the aim of clarifying condition $(iii)$, let us first observe
that both $X_{\mathcal{H}_{0}}(z)$ and $T_z (\mathcal{H}_{0}^{-1}(h))$ are invariant for $P$ and, moreover, 
$X_{\mathcal{H}_{0}}(z) \in T_z (\mathcal{H}_{0}^{-1}(h))$; so, condition $(iii)$ actually deals with the restricted monodromy
$P \colon T_z (\mathcal{H}_{0}^{-1}(h)) \to T_z (\mathcal{H}_{0}^{-1}(h))$. Second (see again \cite[Remark~2.1]{BDF4}), we point out that it is possible to show that any vector $v \in T_z (\pi(\Sigma)) \subset T_z (\mathcal{H}_{0}^{-1}(h))$ is of the form $P\eta + \lambda X_{\mathcal{H}_{0}}(z)$ for some $\lambda \in \mathbb{R}$; thus, the non-degeneracy condition amounts in requiring that the space
\begin{equation*}
\mathcal{F} :=  \bigl{\{} v \in T_z (\mathcal{H}_{0}^{-1}(h)) \colon v = P v + \lambda X_{\mathcal{H}_{0}}(z), \text{ for some } \lambda \in \mathbb{R} \bigr{\}},
\end{equation*}
is contained in $T_z (\pi(\Sigma))$, and thus, that the dimension of $\mathcal{F}$ is equal to the dimension of 
$\pi(\Sigma)$, that is, by condition $(ii)$, the dimension of $\Sigma$.

Concerning the thesis of Theorem~\ref{teo-weinstein0}, let us observe that 
\eqref{loc-theorem} provides a localization information for both the orbit $\{z_i(t)\}_{t \in \mathbb{R}}$ on $\mathfrak{M}$ and the period $\sigma_i$, which can thus be chosen arbitrarily near orbits and periods of the solutions of the unperturbed problem.

In order to prove Theorem~\ref{teo:Efissata}, we are going to apply the abstract Theorem~\ref{teo-weinstein0} to equation \eqref{eq:Efissata-sec}, written in the equivalent Hamiltonian form \eqref{eq-hs00}, where the Hamiltonian $\mathcal{H}_\varepsilon$ is given by \eqref{eq-hamiltoniana} (of course, now the Hamiltonian $\mathcal{H}_\varepsilon$ does not depend on $t$, since the potential $U$ is autonomous). 
The symplectic manifold is thus $\mathfrak{M} =( \mathbb{R}^d \setminus \{0\}) \times \mathbb{R}^d$, endowed with the usual (exact) symplectic form $\omega = \mathrm{d}x \wedge \mathrm{d}p$. Of course, as already observed in Section~\ref{section-2}, the energy $\mathcal{E}_{\varepsilon}$ defined in \eqref{eq:energia2} of a solution $x$ of  \eqref{eq:Efissata-sec} is nothing but the Hamiltonian $\mathcal{H}_\varepsilon$ of the corresponding solution $(x,p)$ of \eqref{eq-hs00}.

To define the periodic manifold $\Sigma \subset \mathfrak{M} \times (0,+\infty)$, we first recall that, given any value $h \in (0,mc^2)$ of the energy, Lemma~\ref{lem:circolare} provides an associated period $\tau = \tau(h)$; accordingly, recalling also Proposition~\ref{topologiaM}, we consider the manifold of circular solutions
\begin{equation*}
\mathcal{M} =
\begin{cases}
\, \mathcal{M}^2_{\tau,+}, &\text{if $d=2$,} \\
\, \mathcal{M}^3_{\tau}, &\text{if $d=3$,}
\end{cases}
\end{equation*}
where, with a slight abuse of notation, throughout the section the above objects are meant as sets of initial conditions
$z(0) = (x(0),p(0)) \in \mathfrak{M}$ giving rise to the associated circular solution (incidentally, let us notice that, differently from Section~\ref{section-3}, here the chosen period $\tau$ is always the minimal one). Finally, let us set
\begin{equation*}
\Sigma = \mathcal{M} \times \{ \tau \} \subset \mathfrak{M} \times (0,+\infty).
\end{equation*}
Notice that $\mathrm{cat}(\Sigma) = \mathrm{cat}(\mathcal{M})$ and, thus, by Remark~\ref{rem-cat} the least integer $m$
greater or equal to $\mathrm{cat}(\Sigma)/2$ is $m=1$ for $d=2$, and $m=2$ for $d=3$. Thus, in the case $d=3$ 
Theorem~\ref{teo-weinstein0} would provide two geometrically distinct periodic solutions $z_1 = (x_1,p_1), z_2 = (x_2,p_2)$ of \eqref{eq-hs00}. However, since the second order problem \eqref{eq:Efissata-sec} is invariant for time-inversion, we cannot exclude that
trivially $z_2(t) = (x_1(-t),-p_1(-t))$. Therefore, in our case Theorem~\ref{teo-weinstein0} will grant the existence of just one periodic solution both for $d=2$ and $d=3$.
 
Clearly, conditions $(i)$ and $(ii)$ of Theorem~\ref{teo-weinstein0} hold true. Therefore, in order to apply the theorem it remains to check that the non-degeneracy condition $(iii)$ is satisfied: this verification, which is of course the crucial one, will be provided in the next two sections, dealing at first with the case $d=3$ and then with the case $d=2$ (where the argument is simpler, and a completely different strategy of proof is possible, see Remark~\ref{rem-continuazione}). Assuming by now that all the assumptions of Theorem~\ref{teo-weinstein0} are satisfied, let us show how to conclude. 

At first, we establish the following continuous dependence result.

\begin{lemma}\label{lem-contdep}
For every $\eta > 0$ there exist $\delta > 0$ and $\bar\varepsilon > 0$ such that, 
if $\vert\varepsilon \vert < \bar\varepsilon$, for every $z^*_0 \in \mathcal{M}$ and for every solution
$z$ of \eqref{eq:Efissata-sec}, it holds that
\begin{equation*}
\vert z(0) - z^*_0 \vert < \delta \; \Rightarrow \; \vert z(t) - z^*(t) \vert < \eta,  \quad \text{for every $t\in[0,\tau+\eta]$,}
\end{equation*}
where $z^*$ is the solution of the unperturbed problem \eqref{eq-hs0} satisfying $z^*(0) = z^*_0$.
\end{lemma}

\begin{proof}
By contradiction, let us suppose that there exists $\hat\eta > 0$ such that, if $\{\delta_n\}_n$ is an arbitrary sequence of positive real numbers with $\delta_n \to 0^+$, there exist $\{\varepsilon_n\}_n$ with $\varepsilon_n \to 0$, $\{z_n\}_n$ solutions of \eqref{eq:Efissata-sec} (with $\varepsilon = \varepsilon_n$) and $\{(z^*_0)_n\}_n \subset \mathcal{M}$ such that
\begin{equation}\label{eq-assurdo}
\vert z_n(0) - (z^*_0)_n \vert < \delta_n \quad \text{ and } \quad \vert z_n(t_n) - z^*_n(t_n) \vert \geq \hat\eta,
\end{equation}
for some $t_n \in [0,\tau+\hat\eta]$, where $z^*_n$ is the solution of the unperturbed problem \eqref{eq-hs0} satisfying $z^*_n(0) = (z^*_0)_n$. By the compactness of the manifold $\mathcal{M}$, we can assume that 
$(z^*_0)_n \to \hat{z}_0$. Hence $z_n(0) \to \hat{z}_0$ as well, and thus, by continuous dependence,
\begin{equation*}
z_n(t) \to \hat{z}(t) \quad \text{ and } \quad z^*_n(t) \to \hat{z}(t), \quad \text{ uniformly on $[0,\tau+\hat\eta]$,}
\end{equation*}
where $\hat{z}$ is the solution of \eqref{eq-hs0} satisfying $\hat z(0) = \hat{z}_0$.
Therefore, the second inequality in \eqref{eq-assurdo} cannot be true.
\end{proof}

Let us fix an arbitrary number $\eta > 0$ and, accordingly, consider
$\delta = \delta(\eta)$ and $\bar\varepsilon= \bar\varepsilon(\eta)$ as given by Lemma~\ref{lem-contdep}.
Assuming without loss of generality that $\delta < \tau$, let us define the set
\begin{equation*}
\mathcal{U}_\delta = \bigl\{ z=(x,p) \in \mathfrak{M} \colon \mathrm{dist}(z,\mathcal{M}) < \delta \bigr\} \times
(\tau-\delta,\tau+\delta),
\end{equation*}
which is of course a neighborhood of $\Sigma$ in $\mathfrak{M} \times (0,+\infty)$.
Then, Theorem~\ref{teo-weinstein0} gives a number $\varepsilon^* = \varepsilon^*(\delta(\eta)) > 0$, which without loss of generality we can take smaller than $\bar\varepsilon(\eta)$, such that, 
if $\vert \varepsilon \vert < \varepsilon^*$, the Hamiltonian system \eqref{eq-hs00} possesses a periodic solution
$z$, with energy $\mathcal{H}_\varepsilon(z) = h$. By \eqref{loc-theorem}, this periodic solution has period $\sigma \in (\tau-\delta,\tau+\delta)$ and satisfies
\begin{equation*}
\mathrm{dist}(z(0),\mathcal{M}) < \delta,
\end{equation*}
that is, in view of the compactness of $\mathcal{M}$, $\vert z(0) - z^*_0 \vert < \delta$ for some $z^*_0 \in \mathcal{M}$. At this point, Lemma~\ref{lem-contdep} can be applied, yielding
\begin{equation*}
\vert z(t) - z^*(t) \vert < \eta,  \quad \text{for every $t\in[0,\tau+\eta]$,}
\end{equation*}
where $z^* = (x^*,p^*)$ is the solution of the unperturbed problem \eqref{eq-hs0} satisfying $z^*(0) = z^*_0$.
In particular
\begin{equation}\label{xifinale}
\vert x(t) - x^*(t) \vert < \eta, \quad \text{for every $t\in[0,\tau+\eta]$.}
\end{equation}
This concludes the proof of Theorem~\ref{teo:Efissata} in the case $d=3$, since, in view of Proposition \ref{topologiaM}, 
$x^*(t) = M x_\tau(t)$ for some $M \in SO(3)$.

In the case $d=2$, again by Proposition \ref{topologiaM}, one would find $x^*(t) = x_\tau(t+\theta)$ for some $\theta \in \mathbb{R}$;
however, it can be assumed $\theta = 0$, up to replacing in \eqref{xifinale} the solution $x$ with $x(\cdot + \theta)$, which of course is  a solution as well.

\subsection{Non-degeneracy in the $3d$-case}\label{section-4.2}

Let us recall that in this case $\mathcal{M} = \mathcal{M}^3_{\tau}$, where $\tau = \tau(h)$ is the period associated with the energy $h$ given by Lemma~\ref{lem:circolare}. Up to a rotation, we can assume that $z^*(0) = (x^*(0),p^*(0)) \in \mathcal{M}^3_{\tau}$ is the initial condition of the solution of 
\eqref{eq-hs0} with
\begin{equation*}
x^*(t) = R (\cos(\omega t),\sin (\omega t),0),
\end{equation*}
where 
\begin{equation*}
\omega = \dfrac{2\pi}{\tau}
\end{equation*} 
and $R$ satisfies \eqref{eq-lem:circolare}.

Denoting by $P_{\mathcal{H}_{0}}$ the monodromy operator at $z^*(0)$ of system \eqref{eq-hs0}, let us introduce the linear space
\begin{equation*}
\mathcal{F} =  \bigl{\{} v \in T_{z^*(0)} (\mathcal{H}_{0}^{-1}(h)) \colon v = P_{\mathcal{H}_{0}} v + \lambda X_{\mathcal{H}_{0}}(z^*(0)), \text{ for some } \lambda \in \mathbb{R} \bigr{\}}.
\end{equation*}
Since, by Proposition~\ref{topologiaM}, $\mathrm{dim}(\mathcal{M}^3_{\tau})=3$, we need to verify that
the dimension of $\mathcal{F}$ is equal to three.

Actually, relying on the discussion in Appendix~\ref{section-appendix}, we are going to perform the symplectic change of coordinates $\Psi$ defined in \eqref{eq-cambiovar4} and already considered in Section~\ref{section-3.3}, so as to verify the corresponding condition in the new coordinates. Accordingly, let us define the linear space
\begin{equation*}
\mathcal{G} :=  \bigl{\{} \eta \in T_{\zeta^*(0)} (\mathcal{K}_{0}^{-1}(h)) \colon \eta = P_{\mathcal{K}_{0}} \eta + \lambda X_{\mathcal{K}_{0}}(\zeta^*(0)), \text{ for some } \lambda \in \mathbb{R} \bigr{\}},
\end{equation*}
where $\zeta^*(0) = \Psi(z^*(0))$, $\mathcal{K}_0$ is the Hamiltonian defined in \eqref{defK03} (and $X_{\mathcal{K}_{0}}$ is the associated vector field), and $P_{\mathcal{K}_{0}}$ is the monodromy operator for system \eqref{eq-hs3} at $\zeta^*(0)$.
Recall that $T_{\zeta^*(0)} (\mathcal{K}_{0}^{-1}(h))$ and $X_{\mathcal{K}_{0}}(\zeta^*(0))$ are invariant for $P_{\mathcal{K}_{0}}$
and that, moreover, $X_{\mathcal{K}_{0}}(\zeta^*(0)) \in T_{\zeta^*(0)} (\mathcal{K}_{0}^{-1}(h))$. In what follows, we are going to denote by 
\begin{equation*}
\tilde P:= P_{\mathcal{K}_{0}}|_{T_{\zeta^*(0)} (\mathcal{K}_{0}^{-1}(h))}
\colon  T_{\zeta^*(0)} (\mathcal{K}_{0}^{-1}(h)) \to T_{\zeta^*(0)} (\mathcal{K}_{0}^{-1}(h))
\end{equation*}
the restricted monodromy operator.

We now claim that the five-dimensional space $T_{\zeta^*(0)} (\mathcal{K}_{0}^{-1}(h))$ can be written as
\begin{equation}\label{claim-final}
T_{\zeta^*(0)} (\mathcal{K}_{0}^{-1}(h)) = \mathrm{ker}( \mathrm{Id} - \tilde P) \oplus V_2,
\end{equation}
where $V_2$ is a two-dimensional subspace invariant for $\tilde P$ and such that
$\mathrm{Id} - \tilde P \colon V_2 \to V_2$ is invertible.

Assuming by now that the claim is true, the conclusion follows easily. Indeed, writing $\eta \in T_{\zeta^*(0)} (\mathcal{K}_{0}^{-1}(h))$
as 
\begin{equation*}
\eta = \eta_1 + \eta_2, 
\qquad 
\text{with $\eta_1 \in \mathrm{ker}( \mathrm{Id} - \tilde P)$, $\eta_2 \in V_2$,}
\end{equation*}
and recalling that $X_{\mathcal{K}_{0}}(\zeta^*(0)) \in \mathrm{ker}( \mathrm{Id} - \tilde P)$,
one has that $\eta \in \mathcal{G}$ if and only if there exists $\lambda \in \mathbb{R}$ such that
\begin{equation*}
( \mathrm{Id} - \tilde P) \eta_1 = \lambda X_{\mathcal{K}_{0}}(\zeta^*(0)) \quad \text{ and }  \quad( \mathrm{Id} - \tilde P) \eta_2 = 0.
\end{equation*}
Now, on the one hand the first equality is satisfied if and only if $\lambda = 0$, with arbitrary 
$\eta_1 \in \mathrm{ker}( \mathrm{Id} - \tilde P)$; on the other hand, since $\mathrm{Id} - \tilde P$ is invertible on $V_2$, the second equality is satisfied if and only if $\eta_2 = 0$. Thus
\begin{equation*}
\mathcal{G} = \mathrm{ker}( \mathrm{Id} - \tilde P).
\end{equation*}
Since, by \eqref{claim-final}, the dimension of $\mathrm{ker}( \mathrm{Id} -\tilde P)$ is three, we deduce that
the dimension of $\mathcal{G}$ is three, as desired.

Let us prove the claim. From the results in Section~\ref{section-3.3}, we know that 
the matrix $\mathbb{P}$ representing the linear operator 
$P_{\mathcal{K}_{0}}$ is $\mathbb{P} = e^{\tau\mathbb{A}}$, with $\mathbb{A}$ the matrix given by \eqref{defA3}.
Moreover, the matrix $\mathbb{A}$ has eigenvalues $\{0, 0, \pm i\omega'', \pm i \omega '\}$ where 
$\omega'' = \omega$ and $\omega' \notin 2\pi\mathbb{Z}/\tau$ (this can be proved as in Section~\ref{section-3.2}, taking into account that the monodromy operator is considered at time $\tau$, that is the \textit{minimal} period of the solution $\zeta^*$).
Notice that the (unique, up to scalar multiplication) eigenvector of $\mathbb{A}$ associated with the eigenvalue $0$ is given by
\begin{equation}\label{formula-campo}
\dot\zeta^*(0) = X_{\mathcal{K}_{0}}(\zeta^*(0)) = (0,\omega,0,0,0,0).
\end{equation}
On the other hand, let us denote by $\eta' = \eta'_{\mathrm{Re}} + i \eta'_{\mathrm{Im}}$ and $\bar{\eta}'$ 
a pair of complex eigenvectors of $\mathbb{A}$ associated with the eigenvalues $\pm i \omega'$ and by $\eta'' = \eta''_{\mathrm{Re}} + i \eta''_{\mathrm{Im}}$ and $\bar{\eta}''$ a pair of complex eigenvectors of $\mathbb{A}$ associated with the eigenvalues $\pm i \omega''$ ($\eta'_{\mathrm{Re}},\eta'_{\mathrm{Im}},\eta''_{\mathrm{Re}},\eta''_{\mathrm{Im}} \in \mathbb{R}^6$). We now show that
\begin{equation}\label{eq-span}
\mathrm{span}\langle \eta'_{\mathrm{Re}},\eta'_{\mathrm{Im}},\eta''_{\mathrm{Re}},\eta''_{\mathrm{Im}} \rangle \subset T_{\zeta^*(0)} (\mathcal{K}_{0}^{-1}(h)).
\end{equation}
Indeed, let us first observe that
\begin{equation*}
T_{\zeta^*(0)} (\mathcal{K}_{0}^{-1}(h)) = \bigl{\{} \eta \in \mathbb{R}^6 \colon \langle \nabla \mathcal{K}_0(\zeta^*(0)), \eta \rangle = 0 \bigr{\}}
\end{equation*}
and that, by \eqref{formula-campo},  $\nabla \mathcal{K}_0(\zeta^*(0)) = (0,0,0,0,\omega,0)$. 
Now, recalling \eqref{defA3} we see that, for both
$\kappa = i\omega'$ and $\kappa = i\omega''$, 
\begin{equation*}
\mathbb{A} \pm \kappa \, \mathrm{Id} = 
\begin{pmatrix}
\pm\kappa & 0 & 0 & A & 0 & 0 \vspace{1pt} \\
B & \pm\kappa & 0 & 0 & C & 0 \vspace{1pt} \\
0 & 0 & \pm\kappa & 0 & 0 & E \vspace{1pt} \\
D & 0 & 0 & \pm\kappa & -B & 0 \vspace{1pt} \\
0 & 0 & 0 & 0 & \pm\kappa & 0 \vspace{1pt} \\
0 & 0 & F & 0 & 0 & \pm\kappa
\end{pmatrix},
\end{equation*}
and thus any (complex) eigenvector of $\mathbb{A}$ associated with $\pm\kappa$ has fifth component equal to zero.
From this fact, \eqref{eq-span} follows.

Since the five vectors $X_{\mathcal{K}_{0}}(\zeta^*(0)), \eta'_{\mathrm{Re}},\eta'_{\mathrm{Im}},\eta''_{\mathrm{Re}},\eta''_{\mathrm{Im}}$ are linearly independent (coming from distinct eigenvalues of $\mathbb{A}$), and recalling that 
$X_{\mathcal{K}_{0}}(\zeta^*(0)) \in T_{\zeta^*(0)} (\mathcal{K}_{0}^{-1}(h))$, we can write
\begin{equation*}
T_{\zeta^*(0)} (\mathcal{K}_{0}^{-1}(h)) = V_3 \oplus V_2,
\end{equation*}
where
\begin{equation*}
V_2 = \mathrm{span} \langle \eta'_{\mathrm{Re}},\eta'_{\mathrm{Im}} \rangle \quad \text{ and } \quad V_3 = \mathrm{span} \langle X_{\mathcal{K}_{0}}(\zeta^*(0)), \eta''_{\mathrm{Re}},\eta''_{\mathrm{Im}} \rangle.
\end{equation*}
The spaces $V_2$ and $V_3$ are obviously invariant for $\mathbb{A}$ and, as a consequence, they are invariant for 
the matrix $\mathbb{P} = e^{\tau\mathbb{A}}$. Moreover, recalling that $\tau\omega' \notin 2\pi\mathbb{Z}$ and $\tau\omega'' = \tau\omega = 2\pi$,
an easy computation shows that 
\begin{equation*}
\mathbb{P}\eta \neq \eta, \; \text{ for every $\eta \in V_2$,} \quad \text{ and } \quad \mathbb{P}\eta = \eta, \; \text{ for every $\eta \in V_3$.}
\end{equation*}
This implies that $\mathrm{Id} - \tilde P \colon V_2 \to V_2$ is invertible and that
$V_3 = \mathrm{ker}( \mathrm{Id} - \tilde P)$, finally proving the claim \eqref{claim-final}.

\subsection{Non-degeneracy in the $2d$-case}\label{section-4.3}

Here, $\mathcal{M} = \mathcal{M}^2_{\tau,+}$, where $\tau = \tau(h)$ is the period associated with the energy $h$.

Following the very same arguments used in Section~\ref{section-4.2}, we can prove the analogous of \eqref{claim-final}, namely
\begin{equation*}
T_{\zeta^*(0)} (\mathcal{K}_{0}^{-1}(h)) = V_1 \oplus V_2,
\end{equation*}
where $V_2$ is as before and $V_1 = \mathrm{ker}( \mathrm{Id} - \tilde P)$ is the one-dimensional space spanned by the vector $X_{\mathcal{K}_{0}}(\zeta^*(0))$.
At this point, the same argument of Section~\ref{section-4.2} shows that the dimension of $\mathcal{G}$
is equal to one, that is the dimension of $\mathcal{M}^2_{\tau,+}$, as desired.

\begin{remark}\label{rem-continuazione}
As already observed in the Introduction, Theorem~\ref{teo:Efissata} in the case $d=2$ is not of bifurcation-type: indeed, what is obtained in this case is the existence of a continuation of the unperturbed periodic orbit $\{x_\tau(t)\}_{t \in \mathbb{R}}$ into a perturbed periodic orbit $\{x(t)\}_{t \in \mathbb{R}}$. As shown above, this result can be established using Theorem~\ref{teo-weinstein0}; however, it also follows from more classical theories.
Indeed, from the discussion in Section~\ref{section-4.2}, and taking into account the fact that $\mathbb{P} = e^{\tau\mathbb{A}}$ is
a symplectic matrix, it is immediate to see that the algebraic multiplicity of $1$ as an eigenvalue of the matrix $\mathbb{P}$ is equal to two, and hence one can apply, for instance, \cite[Theorem~2.4]{MoZe-05}.

We also notice that we have not considered the case $\mathcal{M} = \mathcal{M}^2_{\tau,-}$: indeed, this choice would lead in general to the solution $x(-t)$.
\hfill$\lhd$
\end{remark}

\begin{remark}
Let us notice that a variational formulation for the
fixed-energy problem associated with an equation like \eqref{eq:Efissata-sec} has been recently provided in the paper 
\cite{BoDaMu-23}, relying on the introduction of a suitable relativistic Maupertuis functional. Accordingly, in principle one could try to prove a result like Theorem~\ref{teo:Efissata} by applying Theorem
\ref{thastratto1} to the Maupertuis functional (in the framework of classical mechanics, such a strategy has been used in \cite{AmBe-92}). Within this approach, however, solutions of the fixed-energy problem for \eqref{eq:Efissata-sec} correspond to time-reparameterizations of critical points of the Maupertuis functional: hence, the manifold of circular solutions has to be changed accordingly and we have not investigated in details how the corresponding non-degeneracy condition looks like. Our approach based on the use of Theorem~\ref{teo-weinstein0} for the associated Hamiltonian system seems to be more direct, since it allows us to use the computations of the monodromy operator already developed in Section~\ref{section-3} for the fixed-period problem. On the contrary, it is possible that the use of a variational approach could lead to multiplicity of geometrically distinct solutions, cf.~\cite[Remark~4]{AmBe-92}. 
\hfill$\lhd$
\end{remark}

\appendix 
\section{The monodromy operator and change of variables}\label{section-appendix}

Let us consider the system of differential equations
\begin{equation}\label{ode1}
\dot z = f(z), \quad z \in \mathfrak{M},
\end{equation}
where $\mathfrak{M}$ is a manifold and $f$ is a $\mathcal{C}^1$-vector field on $\mathfrak{M}$.
The associated flow map is denoted by $\varphi^t_f$ (that is, $t \mapsto \varphi^t_f(\xi)$
is the unique - local - solution with $\varphi^0_f(\xi) = \xi$). 

Given a $\mathcal{C}^1$-diffeomorphism $\Psi \colon \mathfrak{M} \to \mathfrak{M}'$ (of course, $\mathfrak{M}'$ is a manifold), let us perform the change of variables 
\begin{equation*}
\zeta = \Psi(z).
\end{equation*} 
As well known (see, for instance, \cite[Chapter 1.1]{MoZe-05}), system \eqref{ode1} is accordingly transformed into the system
\begin{equation}\label{ode2}
\dot \zeta = g(\zeta), \quad \zeta \in \mathfrak{M}',
\end{equation}
where
\begin{equation}\label{cambio-campo}
g(\zeta) = \mathrm{d}\Psi\bigl( \Psi^{-1}(\zeta)\bigr) f\bigl( \Psi^{-1}(\zeta)\bigr).
\end{equation}
Moreover, denoting by $\varphi^t_g$ the flow associated with system \eqref{ode2}, it holds that
\begin{equation}\label{cambio-flusso}
\varphi^t_g \circ \Psi = \Psi \circ \varphi^t_f.
\end{equation}

Let us now suppose that $z \in \mathfrak{M}$ is the initial condition (at time $t=0$) of a periodic solution of system \eqref{ode1} with (not necessarily minimal) period $\tau > 0$, that is
\begin{equation}\label{punto-fisso}
\varphi^{\tau}_f(z) = z.
\end{equation}
The so-called monodromy operator $P_f \colon T_z \mathfrak{M} \to T_z \mathfrak{M}$ is defined as $P_f = \mathrm{d}_\xi \varphi^\tau_f(\xi)|_{\xi = z}$. In the new coordinates, $\zeta = \Psi(z)$ is of course the initial condition of a $\tau$-periodic solution of system \eqref{ode2}.
Denoting by $P_g \colon T_{\zeta} \mathfrak{M}' \to T_{\zeta} \mathfrak{M}'$ the corresponding monodromy operator,
differentiating \eqref{cambio-flusso}, and using \eqref{punto-fisso}, we easily obtain
\begin{equation}\label{cambio-monodromia}
P_g \circ Q = Q \circ P_f,
\end{equation}
where $Q \colon T_z \mathfrak{M} \to T_{\zeta} \mathfrak{M}'$ is the linear map given by $Q = \mathrm{d}\Psi(z)$. Notice that $Q$ is an isomorphism, since $\Psi$ is a diffeomorphism. 

Formula \eqref{cambio-monodromia} shows how the monodromy operator is transformed via a change of variables. A first easy consequence of it is that the spaces of fixed points of $P_f$ and $P_g$ (that is, the space of $\tau$-periodic solutions of the linearizations of \eqref{ode1} and \eqref{ode2} at $z$ and $\zeta$ respectively) are isomorphic as linear spaces and, thus, have the same dimension: we will use this fact in the proof of the non-degeneracy condition $(iii)$ required in Theorem~\ref{thastratto1}, see Section~\ref{section-3.2} and Section~\ref{section-3.3}.

A more subtle consequence will be used in Section~\ref{section-4}. Indeed, in the proof of the non-degeneracy condition $(iii)$ of Theorem~\ref{teo-weinstein0}, we will be concerned with the linear space
\begin{equation*}
\mathcal{F} =  \bigl{\{} v \in T_z (\mathcal{H}_{0}^{-1}(h)) \colon v = P_f v + \lambda f(z), \text{ for some } \lambda \in \mathbb{R} \bigr{\}},
\end{equation*}
where $f(z) = X_{\mathcal{H}_0}(z)$ is an Hamiltonian vector field. In view of \eqref{cambio-monodromia} and recalling \eqref{cambio-campo}, it is easy to see that the space $\mathcal{F}$ is isomorphic, by $Q$, to the space
\begin{equation*}
\mathcal{G} =  \bigl{\{} \eta \in Q \bigl(T_z (\mathcal{H}_{0}^{-1}(h))\bigr) \colon \eta = P_g \eta + \lambda g(\zeta), \text{ for some } \lambda \in \mathbb{R} \bigr{\}}.
\end{equation*}
Moreover, if the diffeomorphism $\Psi$ is symplectic, then $g(\zeta) = \mathcal{K}_0(\zeta)$ with 
$\mathcal{K}_0(\zeta) = \mathcal{H}_0(\Psi^{-1}(\zeta))$ and 
$Q (T_z (\mathcal{H}_{0}^{-1}(h)) ) = T_{\zeta}\mathcal{K}_{0}^{-1}(h)$. 
Summing up, the dimension of the space $\mathcal{F}$ is preserved under symplectic change of coordinates: we will use this fact in Section~\ref{section-4.2}.

\bibliographystyle{elsart-num-sort}
\nocite*{}
\bibliography{BoFePa-biblio}

\end{document}